\documentclass[12pt]{article}
\title{Around definable types in $p$-adically closed fields}
\author{Pablo And\'ujar Guerrero \and Will Johnson}

\usepackage{amsmath, amssymb, amsthm}    	
\usepackage{fullpage} 	
\usepackage{hyperref}
\usepackage[all]{xy}
\usepackage{centernot}

\newcommand{\Qp}{\mathbb{Q}_p}
\newcommand{\Qq}{\mathbb{Q}}
\newcommand{\Rr}{\mathbb{R}}
\newcommand{\Zz}{\mathbb{Z}}
\newcommand{\Nn}{\mathbb{N}}

\DeclareMathOperator{\id}{id}
\DeclareMathOperator{\cl}{cl}

\newcommand{\ba}{{\bar{a}}}
\newcommand{\bb}{{\bar{b}}}

\newcommand{\bx}{{\bar{x}}}
\newcommand{\by}{{\bar{y}}}
\newcommand{\tX}{{\tilde{X}}}

\DeclareMathOperator{\Aut}{Aut}
\DeclareMathOperator{\tp}{tp}
\DeclareMathOperator{\qftp}{qftp}
\DeclareMathOperator{\acl}{acl}
\DeclareMathOperator{\dcl}{dcl}

\newcommand{\df}{\mathrm{def}}
\newcommand{\Mm}{\mathbb{M}}
\newcommand{\eq}{\mathrm{eq}}
\DeclareMathOperator{\vc}{vc}

\newcommand{\pCF}{{p\mathrm{CF}}}

\newcommand{\Oo}{\mathcal{O}}
\newcommand{\mm}{\mathfrak{m}}
\DeclareMathOperator{\val}{val}
\DeclareMathOperator{\res}{res}
\DeclareMathOperator{\rv}{rv}
\DeclareMathOperator{\RV}{RV}

\newcommand{\Ldiv}{\mathcal{L}_{\mathrm{div}}}

\newtheorem{theorem}{Theorem}[section] 
\newtheorem{lemma}[theorem]{Lemma}
\newtheorem{proposition}[theorem]{Proposition}
\newtheorem{corollary}[theorem]{Corollary}
\newtheorem{claim}[theorem]{Claim}
\newtheorem{fact}[theorem]{Fact}

\theoremstyle{definition}
\newtheorem{definition}[theorem]{Definition}

\newtheorem{remark}[theorem]{Remark}
\newtheorem{question}[theorem]{Question}

\newtheorem*{acknowledgment}{Acknowledgments}

\newenvironment{claimproof}[1][\proofname]
               {
                 \proof[#1]
                 
               }
               {
                 \endproof
               }

\begin{document}

\maketitle

\begin{abstract}
  We prove some technical results on definable types in $p$-adically
  closed fields, with consequences for definable groups and definable
  topological spaces.  First, the code of a definable $n$-type (in the field sort)
  can be taken to be a real tuple (in the field sort) rather than an
  imaginary tuple (in the geometric sorts).  Second, any definable type in
  the real or imaginary sorts is generated by a countable union of
  chains parameterized by the value group.  Third, if $X$ is an interpretable set, then the space of global definable types on $X$ is strictly pro-interpretable, building off work of Cubides Kovacsics, Hils, and Ye \cite{cubides-ye,cubides-ye-hils}.  Fourth, global definable types can be lifted (in a non-canonical way) along interpretable surjections.  Fifth, if $G$ is a definable group with definable f-generics (\textit{dfg}), and $G$ acts on a definable set $X$, then the quotient space $X/G$ is definable, not just interpretable.  This explains some phenomena observed by Pillay and Yao \cite{pillay-yao}.  Lastly, we show that interpretable topological spaces satisfy analogues of first-countability and curve
  selection.  Using this, we show that all reasonable notions of definable compactness agree on interpretable topological spaces, and that definable compactness is definable in families.
\end{abstract}

\section{Introduction}

This paper is a collection of closely-related results on $p$-adically
closed fields, related to three themes: definable types, definable
groups, and definable topological spaces.

\subsection{Definable types}
If $M$ is a structure and $q \in S_n(M)$ is a definable type, then $q$
has a ``code'' $\ulcorner q \urcorner$, a tuple (possibly infinite) in
$M^\eq$ which codes the definable sets $\{\bb \in M : \varphi(\bx,\bb) \in
q(\bx)\}$ for each formula $\varphi$.  Our first theorem says that in
$p$-adically closed fields, definable types are coded by \emph{real}
tuples, rather than imaginary tuples:
\begin{theorem}[{= Theorem~\ref{tp-code}}] \label{code0}
  If $K$ is a $p$-adically closed field and $q \in S_n(K)$ is a
  definable type, then $\ulcorner q \urcorner$ is interdefinable with
  a tuple in $K$ (rather than $K^\eq$).
\end{theorem}
The real tuple in Theorem~\ref{code0} is infinite in some cases---see
Proposition~\ref{long}.

Our second main result on definable types shows that they are
generated by small collections of definable families of a specific
form.  First, we fix some conventions.
\begin{remark}\label{type-fam}
In what follows, we identify partial types $\{\varphi_i(x) : i
\in I\}$ with collections of definable sets $\{\varphi_i(M) : i \in
I\}$.  In particular, a partial type $\Sigma(x)$ \emph{extends} a
family $\mathcal{F}$ of definable sets if $\Sigma(x)$ extends the
partial type corresponding to $\mathcal{F}$, meaning that $\Sigma(x)
\vdash x \in D$ for each $D \in \mathcal{F}$.  Likewise, $\mathcal{F}$
\emph{generates} $\Sigma(x)$ if the partial type $\{x \in D : D \in
\mathcal{F}\}$ generates $\Sigma(x)$.
\end{remark}
\begin{definition}\label{refdef}
A family of sets $\mathcal{F}'$ \emph{refines} another
family $\mathcal{F}$ if for any $X \in \mathcal{F}$ there is $Y \in
\mathcal{F}'$ with $Y \subseteq X$.  
\end{definition}
\begin{definition} \label{gfam}
In a $p$-adically closed field, a \emph{$\Gamma$-family} is an interpretable family of the form $\{X_\gamma\}_{\gamma \in \Gamma}$, where $\Gamma$ is the value group,
each $X_\gamma$ is non-empty, and $\gamma' \ge \gamma \implies
X_{\gamma'} \subseteq X_\gamma$.  
\end{definition}
\begin{theorem}[{= Theorems~\ref{tp-gen}, \ref{tp-gen1}}] \label{tp-gen0}
  If $K$ is a $p$-adically closed field and $q \in S_n(K)$ is a
  definable type, then $q$ is generated by the union of countably many
  $\Gamma$-families.  This also holds for definable types in $K^\eq$.
\end{theorem}

Theorem~\ref{tp-gen0} (mostly) falls out of the proof of Theorem~\ref{code0}.
By combining Theorem~\ref{tp-gen0} with work of Simon and
Starchenko~\cite{simon_star_14}, we get a technical result on
definable and interpretable families of sets, which may be of
independent interest:
\begin{theorem}[{= Theorems~\ref{tri}, \ref{tri1}}] \label{tri0}
  Let $\mathcal{F}$ be a definable family of sets in $K \models \pCF$.
  The following are equivalent:
  \begin{enumerate}
  \item \label{t01} $\mathcal{F}$ extends to a definable type $p \in
    S_n(K)$.
  \item $\mathcal{F}$ is refined by a $\Gamma$-family.
  \item \label{t03} $\mathcal{F}$ is refined by a downward directed definable
    family of non-empty sets.
  \end{enumerate}
  The analogous statements also hold in $K^\eq$.
\end{theorem}
The equivalence of (\ref{t01}) and (\ref{t03}) is
analogous to \cite[Theorem~2.12]{types-transversals} in the o-minimal
case.

Next, we extend the results of Cubides Kovacsics, Hils, and Ye
\cite{cubides-ye, cubides-ye-hils} from the real sorts to the
imaginary sorts of $\pCF$.  See Section~\ref{prodsec} for the
definition of uniform definability of definable types, pro-definable
sets, and strictly pro-definable sets.  A \emph{pro-interpretable set}
in a structure $M$ is a pro-definable set in $M^\eq$.
\begin{theorem}[{= Theorem~\ref{uddt-i}}] \label{uddt0}
  The theory $\pCF^\eq$ has uniform definability of definable types.
\end{theorem}
This extends \cite[Theorem~6.3(4)]{cubides-ye} from the home sort to
the imaginary sorts.  If $X$ is an interpretable set, let $X^\df$
denote the space of definable types on $X$.  Theorem~\ref{uddt0} lets
us regard $X^\df$ as a pro-interpretable set, by work of Hrushovski
and Loeser (see Fact~\ref{pro-d}).
\begin{theorem}[{= Theorem~\ref{strict2}}]\label{strict0}
  If $X$ is interpretable, then the pro-interpretable set $X^\df$ is
  strictly pro-interpretable.
\end{theorem}
This extends \cite[Theorem~7.4.8]{cubides-ye-hils} from the home sort
to the geometric sorts.  Meanwhile, Theorem~\ref{code0} can be rephrased in terms
of $X^\df$ as follows:
\begin{theorem}[{= Theorem~\ref{pro-d-literal}}]\label{pro-d-literal0}
  If $X$ is definable (not just interpretable), then $X^\df$ is
  pro-definable (not just pro-interpretable).
\end{theorem}
We also prove the following technical result:
\begin{theorem}[{= Theorem~\ref{lifts}}] \label{lift0}
  Work in a monster model $\Mm$ of $\pCF$.  Let $f : X \to Y$ be an
  interpretable surjection between two interpretable sets.  The
  pushforward map $X^\df \to Y^\df$ is surjective (on $\Mm$-points).
  In other words, any global definable type in $Y$ can be lifted to a
  global definable type in $X$.
\end{theorem}
If we understand correctly, this means that $\pCF^\eq$ has
``surjectivity transfer'' in the sense of
\cite[Definition~2.4.9]{cubides-ye-hils}.  In Theorem~\ref{lift0}, we
don't know whether $M$-definable types can be lifted to $M$-definable
types, when $M$ is not saturated.

\subsection{Definable groups}
Theorem~\ref{code0} has some consequences for
the theory of definable groups.  Recall the following general
definition:
\begin{definition} \label{dfg}
  Work over a monster model $\Mm$.
  An interpretable group $G$ has \emph{definable f-generics}
  (\textit{dfg}) if there is a small model $M_0 \preceq \Mm$ and an
  $M_0$-definable global type $q$ concentrated on $G$, such that every
  left translate of $q$ under the action of $G = G(\Mm)$ is also
  $M_0$-definable.
\end{definition}
Equivalently, $G$ has \textit{dfg} if there is a (global) definable type $q$ on
$G$ with boundedly many left translates.  Pushing $q$ forward along
the map $x \mapsto x^{-1}$, we also get a definable type with
boundedly many right translates.

If one replaces ``$M_0$-definable'' with ``finitely satisfiable in
$M_0$'' in Definition~\ref{dfg}, one gets the notion of \emph{finitely
satisfiable generics} (\textit{fsg}).  In a distal theory like $\pCF$,
the two notions of \textit{dfg} and \textit{fsg} are in some sense
polar opposites.  For example, if $G$ is a 1-dimensional definable
group, then $G$ has exactly one of the two properties.  More
generally, infinite definable groups can have at most one of the two
properties \textit{dfg} and \textit{fsg}.\footnote{The fact that any 1-dimensional definable group has \textit{fsg} or \textit{dfg} is \cite[Lemma~2.9]{pillay-yao}.  The fact that infinite definable groups cannot have both properties simultaneously is alluded to in the proof of \cite[Corollary~2.11]{pillay-yao}, but not explained there, as far as we can tell.  This fact can be proved as follows.  Suppose $G$ has \textit{dfg} and \textit{fsg}.  Take a global types $q_1$ and $q_2$ and a small model $M$ such that every right-translate of $q_1$ is $M$-definable, and every left-translate of $q_2$ is finitely satisfiable in $M$.  In an elementary extension of the monster model $\Mm$, take $(a,b)$ realizing $q_1 \otimes q_2$, meaning that $a$ realizes $q_1$ and $b$ realizes $q_2 | \Mm a$.  Definable types commute with finitely satisfiable types \cite[Lemma~3.4]{HP}, so $(b,a)$ realizes $q_2 \otimes q_1$.  In particular, $a$ realizes $q_1 | \Mm b$.  Now one can see easily that $\tp(a \cdot b / \Mm)$ is the right-translate of $q_1 = \tp(a/\Mm)$ by $b$, so $\tp(a \cdot b / \Mm)$ is $M$-definable.  Similarly, $\tp(a \cdot b / \Mm)$ is the left-translate of $q_2 = \tp(b/\Mm)$ by $a$, so $\tp(a \cdot b / \Mm)$ is finitely satisfiable in $M$.  Then $\tp(a \cdot b / \Mm)$ is generically stable \cite[Remark~3.3]{HP}.  By \cite[Proposition~2.27]{distal}, generically stable types in distal theories are constant (i.e., realized), and so $\tp(a \cdot b / \Mm)$ must be constant. But then $q_1$ and $q_2$ are themselves constant types, and have unboundedly many translates, a contradiction.}   It turns out that a
definable group $G$ has \textit{fsg} if and only if it is definably
compact \cite{O-P,johnson-fsg}.  Therefore, \textit{dfg} can be
understood as the ``opposite'' of definable compactness.

Using Theorem~\ref{code0}, we can prove the following:
\begin{theorem}[{= Theorem~\ref{quot}}] \label{quot0}
  In a model of $\pCF$, suppose $G$ is an interpretable \textit{dfg} group and $X$ is
  a definable set with a definable action of $G$.  Then the quotient
  space $X/G$ (the set of orbits) is definable.  More precisely, the
  interpretable set $X/G$ is in interpretable bijection with a
  definable set.
\end{theorem}
This allows us to answer a question asked by Pillay and Yao \cite{pillay-yao}:
\begin{corollary}[{= Corollary~\ref{dim0-quot}}]
  Suppose $G$ is a definable group in a $p$-adically closed field and
  $H$ is a definable open subgroup.  If $H$ has \textit{dfg}, then $H$
  has finite index in $G$.
\end{corollary}
Here, $H$ is ``open'' with respect to the definable manifold structure
on $G$ respecting the group topology, constructed by Pillay
\cite{pillay}.

In work by the second author and Yao \cite{jy-abelian}, Theorem~\ref{quot0} is applied to prove new results about abelian definable groups in
$p$-adically closed fields, including a decomposition of abelian definable groups into \textit{dfg} and \textit{fsg} components.

\subsection{Definable and interpretable topological spaces}
A \emph{definable topology} is a topology with a (uniformly) definable
basis of open sets, and a \emph{definable topological space} is a
definable set with a definable topology.  An \emph{interpretable
topological space} in $M$ is a definable topological space in $M^\eq$.
Definable and interpretable topological spaces arise naturally when studying definable and interpretable groups in $\pCF$.  In fact, every definable group  or interpretable group  has the structure of a definable or interpretable topological space in a canonical way \cite[Lemma~3.8]{pillay} \cite[Theorem~1.2]{johnson-admissible}.

Theorems~\ref{tp-gen0} and \ref{tri0} have consequences for definable
and interpretable topological spaces in $p$-adically closed fields.
In particular we show that these satisfy a definable analogue of first
countability, namely: every point has a neighborhood basis that is a
$\Gamma$-family (Proposition~\ref{prop:pCF_1st_countability}).  Over
$\Qp$, this implies actual first countability.  Using definable first
countability, we get a form of definable curve selection
(Lemma~\ref{curve-selection}).  We also show the equivalence of
various notions of ``definable compactness'' as follows.

\begin{theorem}[{= Theorem~\ref{thm:compactness}}]\label{thm:compactness-intro}
Let $(Z,\tau)$ be a definable (or interpretable) topological space in a $p$-adically closed field $K$. The following are equivalent. 
\begin{enumerate}
    \item Every downward directed interpretable family of non-empty closed sets in $Z$ has non-empty intersection.  
    \item Every interpretable $\Gamma$-family of closed sets in $Z$ has non-empty intersection.
    \item Every interpretable curve $f:D\subseteq K\rightarrow Z$ has a converging interpretable restriction.  (See Definition~\ref{curve}).
    \item Every definable type in $S_Z(K)$ has a specialization.  (See
      Definition~\ref{spec}.)
    \item Every interpretable family of closed sets $\{C_y\}_{y\in Y}$ in $Z$ with the finite intersection property has a finite transversal, i.e. there exists a finite set $T$ with $T\cap C_y\neq \varnothing$ for every $y\in Y$.
\end{enumerate}
\end{theorem}
Say that an interpretable topological space is \emph{definably
compact} if it satisfies these equivalent conditions.  We prove two
additional, useful facts about definable compactness:
\begin{theorem}[{= Theorem~\ref{qp-case}}]\label{qp-case0}
  An interpretable topological space $Z$ in $\Qp$ is definably compact
  if and only if it is compact.
\end{theorem}
\begin{theorem}[{= Theorem~\ref{compactness-def}}] \label{compactness-def0}
  Definable compactness is definable in families:
    Let $\{X_b : b \in Y\}$ be an interpretable family of interpretable
  topological spaces in $\Mm$.  Then $\{b \in Y : X_b \text{ is
    definably compact}\}$ is interpretable.
\end{theorem}
Note that Theorems~\ref{qp-case0} and \ref{compactness-def0} uniquely
characterize definable compactness---it is the only
automorphism-invariant definable concept extending genuine compactness
on $\Qp$.  This further supports the idea that we have isolated the
correct notion of ``definable compactness.''

\subsection{Outline}
In Sections~\ref{tools} and \ref{analysis} we analyze definable
$n$-types, proving Theorem~\ref{code0} on real codes, and the
non-imaginary case of Theorem~\ref{tp-gen0} on generation by
$\Gamma$-families.  Subsection~\ref{infinite} gives an example of a
definable $n$-type whose code is an infinite tuple, rather than a
finite tuple.  In Section~\ref{group-app}, we apply
Theorem~\ref{code0} to definable groups, proving Theorem~\ref{quot0}.

In Section~\ref{ss} we combine Theorem~\ref{tp-gen0} with results of
Simon and Starchenko \cite{simon_star_14} to prove the non-imaginary
case of Theorem~\ref{tri0} on definable families.  In
Section~\ref{prodsec} we review the machinery of uniform definability
of definable types, pro-definable sets, and strict pro-definability,
which will be applied later in the paper.  While discussing this, we
deduce Theorem~\ref{pro-d-literal0} from Theorem~\ref{code0}.  In
Section~\ref{imaginary}, we extend the facts of
Sections~\ref{ss}--\ref{prodsec} from $\pCF$ to $\pCF^\eq$, proving
Theorems~\ref{uddt0}--\ref{strict0} on strict pro-definability, and
the technical Theorem~\ref{lift0} on lifting.  In the process, we get
the imaginary cases of Theorems~\ref{tp-gen0} and \ref{tri0}.

In Section~\ref{topsec} we apply everything to the study of definable
and interpretable topological spaces, proving definable first
countability, curve selection, and
Theorems~\ref{thm:compactness-intro}--\ref{compactness-def0}.
If one is only interested in definable topological spaces, rather than
interpretable topological spaces, then Section~\ref{imaginary} can be
skipped, and Section~\ref{prodsec} is only used in the proof of
Theorem~\ref{compactness-def0}.

We also include an appendix containing the proofs of three known
results due to Delon, Cubides Kovacsics, Ye, and Hils
\cite{delon,cubides-ye,cubides-ye-hils}, specifically the following:
\begin{itemize}
\item The definability of types over $\Qp$ (Fact~\ref{delon}).
\item The uniform definability of definable types in $\pCF$ (Fact~\ref{uddt}).
\item The strict pro-definability of $X^\df$ (Fact~\ref{strict}).
\end{itemize}
These three facts play an important role in the paper, and have short
proofs from the machinery in Sections~\ref{tools}--\ref{ss}, so it
seemed worthwhile to include the proofs.  In the case of strict
pro-definability, our proof seems to be novel.

\subsection{Notation}
We use letters like $M$ and $K$ to denote models, and we use
$\Mm$ to denote monster models.  We distinguish $M$ from
$M^{\eq}$.  We write $\dcl^{\eq}$ for definable closure in
$M^{\eq}$, and $\dcl$ for definable closure in $M$, and similarly with
$\acl^{\eq}$ and $\acl$ for algebraic closure.  We distinguish ``definable sets'' (in $M$) from ``interpretable sets'' (in $M^\eq$).  Generally speaking, we use ``interpretable in $M$'' as a synonym for ``definable in $M^\eq$'' when referring to sets, functions, subsets, and families.  However, we say ``definable type'' rather than ``interpretable type'' when referring to definable types in $M^\eq$.

We call elements of $M$ ``reals'' and
elements of $M^\eq$ ``imaginaries.''  A ``real tuple'' is a tuple of
reals, possibly infinite.  An ``imaginary tuple'' is a tuple of
imaginaries, possibly infinite.  We write tuples as $a, b, x, y, \ldots$ rather than $\ba, \bb, \bx, \by, \ldots$.  
We write $\varphi(M)$ rather than $\varphi(M^{|x|})$ for the set defined by an $\mathcal{L}(M)$-formula $\varphi(x)$.  If $x$ is a tuple of variables, then $M^x$ denotes $M^{|x|}$, and $S_x(A)$ denotes the type space $S_{|x|}(A)$.  If $X$ is an $A$-definable or $A$-interpretable set, then $S_X(A)$ denotes the set of types over $A$ concentrating on $X$.

We distinguish ``families of definable sets'' from ``definable families of sets.''  The former is an arbitrary collection of sets, each of which is definable, without any uniform definability.  In contrast, a ``definable family of sets'' is a uniformly definable family of sets parameterized in a definable way by a definable index set.  That is, a ``definable family of sets'' is a family of the form $\{D_i\}_{i \in X}$ such that $\{(i,j) : i \in X, ~ y \in D_i\}$ is a definable set.  An ``interpretable family of sets'' is a definable family of sets in $M^\eq$.

A ``type'' means a complete type, by default.  A ``partial type'' means a partial type.  A ``quantifier-free type'' means a complete quantifier-free type, that is, a partial type of the form $\qftp(a/B)$ for some tuple $a$ and set $B$.  ``Definable'' always means ``definable with parameters.''  We abbreviate $\varnothing$-definability (definability without parameters) as ``$0$-definability.''

We let $S^{\df}_n(M)$ denote the space of definable $n$-types over a model $M$, and we define $S^\df_x(M)$ and $S^\df_X(M)$ similarly for a tuple of variables $x$ or a definable set $X$.
We write $\ulcorner X \urcorner$ for the code in $M^\eq$ of a
definable or interpretable set $X$, and we write $\ulcorner p \urcorner$ for the code
of a definable type $p \in S^{\df}_n(M)$.  We can take $\ulcorner X
\urcorner$ to be a single element in $M^\eq$, but $\ulcorner p
\urcorner$ potentially needs to be an infinite tuple.

We identify a partial type $\{\varphi_i(x,b_i) : i \in I\}$ with the corresponding family of definable sets $\{\varphi_i(M,b_i) : i \in I\}$, and vice versa.  Thus $D \in p$ means that $p(x) \vdash x \in D$.  Similarly, if $\mathcal{F}$ is a family of definable sets, then $p \supseteq \mathcal{F}$ means that $D \in \mathcal{F} \implies p(x) \vdash x \in D$.

We denote languages with symbols like $\mathcal{L}$, preferring to use $L$ for fields.  If $M$ is an $\mathcal{L}$-structure and $A \subseteq M$, then $\mathcal{L}(A)$ denotes the language obtained by adding elements of $A$ as constants.  If $T$ is an $\mathcal{L}$-theory, then $T_\forall$ is the set of
universally quantified $\mathcal{L}$-sentences implied by $T$.  A
model of $T_\forall$ is a
substructure of a model of $T$.

If $K$ is a valued field, we write the valuation as $\val(-)$, the
value group as $\Gamma$ or $\Gamma(K)$, the valuation ring as $\Oo$,
the maximal ideal as $\mm$, the residue field $\Oo/\mm$ as $k$, and
the residue map as $\res : \Oo \to k$.  We write valuations
additively.

In this paper, a \emph{$p$-adically closed field} is a field
elementarily equivalent to $\Qp$, and the theory of $p$-adically
closed fields is written $\pCF$.  One could also consider the theory
$\pCF_d$ of $p$-adically closed fields of $p$-rank $d$, i.e., fields
elementarily equivalent to finite extensions $K/\Qp$ with $[K : \Qp] =
d$.  We will not consider such fields here, but we believe all the
results in this paper generalize\footnote{In order to get quantifier elimination in the Macintyre language and definable Skolem
functions, one needs to name some elements of the
prime model as constants.}
from $\pCF = \pCF_1$ to $\pCF_d$.

Let $\Ldiv$ be the language of valued fields with the divisibility
predicate $\val(x) \le \val(y)$.  The theory ACVF has quantifier elimination in $\Ldiv$, but
$\pCF$ does \emph{not}.  On the other hand, Macintyre showed that $\pCF$ has quantifier
elimination if we add a predicate $P_m$ for the set of $m$th powers, for
each $m$ \cite{macintyre}.  We will consider $\pCF$ as an
$\Ldiv$-theory.  Therefore, a ``quantifier-free type'' means a
quantifier-free type in the language $\Ldiv$.  (In
Section~\ref{qftp-tp} we will consider a larger language
$\mathcal{L}_\chi$ and write $\pCF^\chi$ for the theory of $\pCF$ in
this language.)

\begin{acknowledgment}
The first author was supported by the Fields Institute for
Research in Mathematical Sciences, specifically by the 2021 Thematic Program on Trends in Pure and Applied Model Theory and the 2022 Thematic Program on Tame Geometry, Transseries and Applications to Analysis and Geometry.  The second author was supported by the National Natural Science
  Foundation of China (Grant No.\@ 12101131).  Parts of this paper
  arose out of joint work with Ningyuan Yao, to appear in
  \cite{jy-abelian}.  Vincent Ye and Silvain Rideau provided some
  helpful information about $p$-adically closed fields, including the
  reference to \cite[Lemma~5.11]{cubides-ye}.
\end{acknowledgment}

\section{Tools} \label{tools}
\subsection{Interdefinability and interalgebraicity}
Let $\Mm$ be a monster model of $p$CF.  Recall that ``tuples" can be infinite by default.
\begin{lemma} \label{interdef}
  Let $a$ be an imaginary tuple (in $\Mm^\eq$) and $b$ be a real tuple (in $\Mm$).
  \begin{enumerate}
  \item \label{id1} If $a \in \acl^\eq(b)$, then $a \in
    \dcl^\eq(b)$.
  \item \label{id2} If $a \in \acl^\eq(b)$ and $b \in \acl^\eq(a)$, then
    there is a real tuple $c$ such that $a \in \dcl^\eq(c)$ and
    $c \in \dcl^\eq(a)$.
  \end{enumerate}
\end{lemma}
\begin{proof}
  The proof uses two facts.  First, $p$CF (but not $p\mathrm{CF}^\eq$)
  has definable Skolem functions.  Second, any finite subset of
  $\Mm^n$ is coded by a tuple in $\Mm$ (this holds in any theory of
  fields, and can be deduced from elimination of
  imaginaries in ACF, for example).
  \begin{enumerate}
  \item By definable Skolem functions,
    $\dcl(b)$ is a model $M \preceq \Mm$, and then
    $\dcl^\eq(b) = M^\eq \preceq \Mm^\eq$.  This implies
    $\dcl^\eq(b)$ is algebraically closed within $\Mm^\eq$ (as
    models are always algebraically closed), and so $\acl^\eq(b)
    = \dcl^\eq(b)$.
  \item It suffices to show for each finite subtuple $a_0
    \subseteq a$, there is a finite tuple $e$ of reals
    with $a_0 \in \dcl^\eq(e) \subseteq
    \dcl^\eq(a)$.  Fix such a finite subtuple $a_0
    \subseteq a$.  By (\ref{id1}), $a_0 \in \dcl^\eq(b)$, so
    there is a finite tuple $b_0 \subseteq b$ with
    $a_0 \in \dcl^\eq(b_0)$.  Write $a_0 =
    f(b_0)$ for some 0-interpretable function $f$.  Let $S$ be the
    orbit of $b_0$ under $\Aut(\Mm/a)$.  If $c \in
    S$, then $c \equiv_{a} b_0$, so $c
    a_0 \equiv b_0 a_0$.  In particular, $a_0
    = f(c)$ for any $c \in S$, which implies
    $a_0 \in \dcl^\eq(\ulcorner S \urcorner)$, where $\ulcorner
    S \urcorner$ is the code for $S$.  Clearly $\ulcorner S
    \urcorner \in \dcl^\eq(a)$, since $S$ is an
    $a$-definable finite set.  Take $e$ to be a real tuple
    interdefinable with $\ulcorner S \urcorner$, using the second fact
    mentioned above. \qedhere
  \end{enumerate}
\end{proof}
\begin{remark} \label{interdef2}
  Lemma~\ref{interdef} also holds after naming a set $E$ of real
  parameters.  In particular, if $a$ is an imaginary tuple and $b$ is
  a real tuple, then the following hold.
  \begin{enumerate}
  \item[$1'.$] If $a \in \acl^\eq(Eb)$, then $a \in \dcl^\eq(Eb)$.
  \item[$2'.$] If $a \in \acl^\eq(Eb)$ and $b \in \acl^\eq(Ea)$, then there
    is a real tuple $c$ such that $a \in \dcl^\eq(Ec)$ and $c \in
    \dcl^\eq(Ea)$.
  \end{enumerate}
  In fact, the proof of Lemma~\ref{interdef} continues to apply after
  naming parameters.  Alternatively, if we view $E$ as an infinite
  tuple, then (1$'$) is an instance of Lemma~\ref{interdef}(\ref{id1}), and
  (2$'$) follows by applying Lemma~\ref{interdef}(\ref{id2}) to the
  interdefinable tuples $Ea$ and $Eb$ (if $Ea$ is interdefinable with
  $c$, then $Ea$ is interdefinable with $Ec$).
\end{remark}

\subsection{Completions of quantifier-free types} \label{qftp-tp}
As mentioned in the introduction, we consider $\pCF$ as an
$\Ldiv$-theory, and so a ``quantifier-free type'' means a
quantifier-free type in the language $\Ldiv$.
\begin{proposition} \label{qf-count}
    Suppose $K \models \pCF$ and $q$ is a quantifier-free $n$-type over $K$.
  Then $q$ has at most $2^{\aleph_0}$ completions in $S_n(K)$.
\end{proposition}
This is apparently well-known to experts, but we had difficulty locating a proof
in the literature, so we include one here for completeness.  If $K
\models \pCF$ and $m \ge 1$, then
\begin{equation*}
  K^\times/(K^\times)^m \cong \Qp^\times/(\Qp^\times)^m =: Q_m,
\end{equation*}
because $\Qp^\times/(\Qp^\times)^m$ is finite.  Let
\begin{equation*}
  \chi_m : K^\times \to Q_m
\end{equation*}
be the natural map.  Let $\mathcal{L}_{\chi}$ be the expansion of
$\Ldiv$ by unary predicates $P_{m,c}$ for $m \ge 1$ and $c \in Q_m$,
where $P_{m,c}(x)$ is interpreted as $\chi_m(x) = c$.  Note that
$P_{m,1}$ is the Macintyre predicate (the predicate naming the $m$th
powers), and so $\pCF$ has quantifier elimination in the language
$\mathcal{L}_{\chi}$.  We will write $\pCF^\chi$ for $\pCF$ as an
$\mathcal{L}_{\chi}$-theory.  If $K \models \pCF^\chi_\forall$, then the $P_{m,c}$ predicates come from homomorphisms $\chi_m : K^\times \to Q_m$, but the kernel of $\chi_m$ needn't be $(K^\times)^m$.

Recall the following definition:
\begin{definition}
  If $K$ is an unramified valued field of mixed characteristic $(0,p)$
  and $n \ge 0$, then $\RV_n(K) = K^\times/(1 + p^{n+1}\Oo_K)$, and
  $\rv_n : K^\times \to \RV_n(K)$ denotes the natural map. 
\end{definition}
There is a natural homomorphism $\RV_n(K) \to \Gamma(K)$ with kernel
$\Oo_K^\times/(1 + p^{n+1}\Oo_K)$ or equivalently
$(\Oo_K/p^{n+1}\Oo_K)^\times$.  If $L/K$ is a valued field extension, then
there is a natural injection $\RV_n(K) \hookrightarrow \RV_n(L)$.
\begin{lemma} \label{mac-rv}
  There is a function $f : \Nn \to \Nn$ with the following properties:
  \begin{enumerate}
  \item If $K \models \pCF$ and $x \in 1 + p^{f(n)+1}\Oo_K$, then $x$ is an
    $n$th power.
  \item If $K \models \pCF^\chi_\forall$, then $\chi_n : K^\times \to Q_n$
    factors through $\rv_{f(n)} : K^\times \to \RV_{f(n)}(K)$.
  \end{enumerate}
\end{lemma}
\begin{proof}
  \begin{enumerate}
  \item If $f$ works for $\Qp$ then it works for any $K \equiv
    \Qp$, so we may assume $K = \Qp$.  When $K = \Qp$, we only
    need to show that $1$ is in the interior of the set of $n$th
    powers.  This follows easily by Hensel's lemma or dimension theory.
  \item Embed $K$ into a model $L \models \pCF^\chi$.  The $P_{n,c}$
    predicates on $L$ extend those on $K$, so $\chi_n : L \to Q_n$
    extends $\chi_n : K \to Q_n$.  If $x \in K$ and $\rv_{f(n)}(x) =
    1$, then $\chi_n(x) = 1$ by the first part applied to $L$. \qedhere
  \end{enumerate}
\end{proof}

\begin{lemma} \label{qf-count0}
  Suppose $K \models \pCF$ and $L/K$ is a finitely generated extension
  of valued fields.  Then there are at most $2^{\aleph_0}$ expansions
  of $L$ to a model of $\pCF^\chi_\forall$.
\end{lemma}
\begin{proof}
  Note that any $\pCF^\chi_\forall$-structure on $L$ must extend the
  natural $\pCF^\chi$-structure on $K$.  Indeed, taking a model $M
  \models \pCF^\chi$ extending $L$, we see that the $\chi_n$ functions
  on $M$ (and therefore $L$) extend those on $K$, because $M \succeq K$ by model
  completeness of $\pCF$ in $\Ldiv$.

  It suffices to show for fixed $m$ that there are at most
  $2^{\aleph_0}$ ways of extending $\chi_m$ from $K$ to $L$.  By
  Lemma~\ref{mac-rv}, $\chi_m$ factors through $\rv_n$ for some $n = f(m)$
  not depending on $K$ or $L$.  As $\chi_m$ is a homomorphism to the
  finite group $Q_m$, it suffices to show that $\RV_n(L)/\RV_n(K)$ is
  countable for each $n$.

  As $\Gamma(K)$ is a $\Zz$-group, $\Gamma(K)$ has countable index in
  its divisible hull $\Gamma(K) \otimes_\Zz \Qq$.  The $\Qq$-vector
  space $(\Gamma(L) \otimes_\Zz \Qq)/(\Gamma(K) \otimes_\Zz \Qq)$ has
  finite dimension by Abhyankar's inequality.  Then $\Gamma(K)$ and
  $\Gamma(K) \otimes_\Zz \Qq$ have countable index in $\Gamma(L)
  \otimes_\Zz \Qq$.  A fortiori, $\Gamma(K)$ has countable index in
  $\Gamma(L)$.

  If $L \not \models \pCF_\forall$ then there are zero expansions of $L$ to a model of $\pCF^\chi_\forall$, and there is nothing to prove.  Suppose $L \models \pCF_\forall$, and take a model $M \models \pCF$ extending $L$.  As above, $M \succeq K$ by model completeness of $\pCF$.  Then the inclusions
  \begin{equation*}
    \Oo_K/p^{n+1}\Oo_K \to \Oo_L/p^{n+1}\Oo_L \to \Oo_M/p^{n+1}\Oo_M
  \end{equation*}
  must be isomorphisms, because the leftmost ring is finite and the
  rightmost ring is an elementary extension of it.  Finally, the diagram
  \begin{equation*}
    \xymatrix{
      1 \ar[r] & (\Oo_K/p^{n+1}\Oo_K)^\times \ar[r] \ar@{=}[d] & \RV_n(K) \ar[r] \ar[d] & \Gamma(K) \ar[r] \ar[d] & 1 \\
      1 \ar[r] & (\Oo_L/p^{n+1}\Oo_L)^\times \ar[r]  & \RV_n(L) \ar[r]  & \Gamma(L) \ar[r]  & 1
    }
  \end{equation*}
  shows that $\RV_n(K)$ has countable index in $\RV_n(L)$.
\end{proof}
Finally, Proposition~\ref{qf-count} follows formally from
Lemma~\ref{qf-count0}, using quantifier elimination in $\pCF^\chi$.
Next, we deduce some consequences of Proposition~\ref{qf-count}.  In
the following lemmas, if $q$ is a (complete) type over a model of
$\pCF$, then $\hat{q}$ denotes its quantifier-free part.

\begin{lemma} \label{qf-def}
  Suppose $K \models \pCF$ and $q \in S_n(K)$.  If $\hat{q}$ is
  definable, then $q$ is definable.
\end{lemma}
\begin{proof}
  It suffices to show that $q$ has at most $2^{\aleph_0}$-many heirs
  over any elementary extension $L/K$.  Suppose $p \in S_n(L)$ is an
  heir of $q$.  The fact that $p$ is an heir of $q$ implies that
  $\hat{p}$ is definable, defined by the same definition schema as
  $\hat{q}$.  In particular, $\hat{p}$ is the same across all heirs $p
  \in S_n(L)$.  By Proposition~\ref{qf-count}, there are at most
  $2^{\aleph_0}$-many heirs.
\end{proof}

\begin{lemma} \label{qf-def-2}
  Suppose $K_0 \preceq K \models \pCF$ and $q \in S_n(K)$.  Suppose
  $\hat{q}$ is $K_0$-definable.
  \begin{enumerate}
  \item \label{qd1} $q$ is $K_0$-definable.
  \item \label{qd2} $(q \restriction K_0)(x) \cup \hat{q}(x) \vdash q(x)$.
  \end{enumerate}
\end{lemma}
\begin{proof}
  \begin{enumerate}
  \item By Lemma~\ref{qf-def}, $q$ is $K$-definable.  Embed $K$ into
    a monster model $\Mm$.  Let $q^\Mm \in S_n(\Mm)$ be the heir of
    $q$ over $\Mm$.  Then $q^\Mm$ is definable with the same
    definition as $q$.  The quantifier-free part $\widehat{q^\Mm}$ is
    $K_0$-definable, so it is $\Aut(\Mm/K_0)$-invariant.  By
    Proposition~\ref{qf-count}, $q^\Mm$ has at most
    $2^{\aleph_0}$-many images under $\Aut(\Mm/K_0)$.  If $D$ is a
    definable set with a small number of images under $\Aut(\Mm/K_0)$,
    then $D$ must be $\acl^{\eq}(K_0)$-definable, hence
    $K_0$-definable, as $K_0$ is a model.  Therefore $q^\Mm$ is
    $K_0$-definable.
  \item Take $b \in \Mm^n$ realizing $q \restriction K_0$ and
    $\hat{q}$.  We claim $\tp(b/K) = q$.  Let $r =
    \tp(b/K)$.  Then $\hat{r} = \qftp(b/K) = \hat{q}$,
    which is $K_0$-definable.  By part (\ref{qd1}), $r$ is $K_0$-definable.
    Additionally, $r \restriction K_0 = \tp(b/K_0) = q
    \restriction K_0$.  Since $r$ and $q$ are $K_0$-definable types
    with the same restriction to $K_0$, they must be equal, and so
    $\tp(b/K) = r = q$. \qedhere
  \end{enumerate}
\end{proof}

\subsection{Valued vector spaces}
Let $K$ be a valued field.  The following definition appears in
\cite{hrushovski,johnson}, among other places.
\begin{definition} \label{vvs}
  Let $V$ be a $K$-vector space.  A \emph{valued vector space
  structure} or (\emph{VVS structure}) on $V$ consists of the
  following data:
  \begin{enumerate}
  \item A linearly ordered set $\Gamma(V)$.
  \item An action of $\Gamma$ on $\Gamma(V)$
    \begin{equation*}
      + : \Gamma \times \Gamma(V) \to \Gamma(V),
    \end{equation*}
    strictly increasing in each variable.
  \item A surjective function $\val : V \setminus \{0\} \to \Gamma(V)$
    satisfying the axioms
    \begin{gather*}
      \val(av) = \val(a) + \val(v) \tag{for $a \in K$, $v \in V$} \\
      \val(v + w) \ge \min(\val(v),\val(w)) \tag{for $v, w \in V$}
    \end{gather*}
    where we formally extend $\val$ to $0 \in V$ by taking $\val(0) =
    +\infty > \Gamma(V)$.
  \end{enumerate}
  We identify two VVS structures if they induce the same divisibility
  relation $\val(v) \le \val(w)$ on $V$.  A \emph{valued $K$-vector
  space} is a $K$-vector space with a VVS structure.  If $V$ is a
  definable $K$-vector space, then a VVS structure on $V$ is
  \emph{definable} if the divisibility relation $\val(v) \le \val(w)$
  is definable.
\end{definition}
\begin{remark} \label{abhyankar}
  If $V$ is a valued $K$-vector space and $\dim_K(V) = n < \infty$,
  then there are at most $n$ orbits of $\Gamma$ in $\Gamma(V)$, by
  \cite[Section~2.5]{hrushovski} or \cite[Remark~2.2]{johnson}.  In fact, if
  $v_1,\ldots,v_m \in V$ are such that $\val(v_1),\ldots,\val(v_m)$
  are in distinct orbits of $\Gamma$, then $v_1,\ldots,v_m$ are
  $K$-linearly independent.
\end{remark}
\begin{remark}
  If $\dim_K V = 1$, then there is a unique VVS structure on $V$.
\end{remark}

\begin{definition}
  Let $V$ be an $n$-dimensional valued $K$-vector space.  A
  \emph{splitting basis} is a set $\{v_1,\ldots,v_n\} \subseteq V$
  such that
  \begin{equation*}
    \Gamma + \val(v_1) > \Gamma + \val(v_2) > \cdots > \Gamma + \val(v_n).
  \end{equation*}
  A valued $K$-vector space or a VVS structure is \emph{split} if a
  splitting basis exists.
\end{definition}
If $\{v_1,\ldots,v_n\}$ is a splitting basis, then
$\{v_1,\ldots,v_n\}$ is a basis and every coset has the form $\Gamma +
\val(v_i)$, by Remark~\ref{abhyankar}.
\begin{definition}
  If $V$ is an $n$-dimensional $K$-vector space, a \emph{complete
  filtration} is a chain of $K$-linear subspaces of length $n+1$:
  \begin{equation*}
    0 = V_0 \subset V_1 \subset \cdots \subset V_{n-1} \subset V_n = V.
  \end{equation*}
\end{definition}
\begin{proposition} \label{filter}
  Let $V$ be a split $n$-dimensional valued $K$-vector space.  Let
  $C_1,\ldots,C_n$ be the cosets of $\Gamma$, ordered so that $C_1 >
  C_2 > \cdots > C_n$.  For $0 \le i \le n$, let $V_i = \{0\} \cup \{x
  \in V : \val(x) \in \bigcup_{j = 1}^i C_j\}$.  Then $V_i$ is a
  $K$-linear subspace, $\{V_i\}_{0 \le i \le n}$ is a complete
  filtration, and the VVS structure on $V$ is determined by the
  filtration $\{V_i\}_{0 \le i \le n}$.
\end{proposition}
\begin{proof}
  Take a splitting basis $\{v_1,\ldots,v_n\}$.  By the ultrametric
  inequality, $V_i$ is the $K$-linear span of $\{v_1,\ldots,v_i\}$, so
  $\{V_i\}_{0 \le i \le n}$ is a complete filtration.  The ultrametric
  inequality also shows that $\val : V_i \setminus V_{i-1} \to C_i$ is
  induced by a VVS structure on $V_i/V_{i-1}$.  As $V_i/V_{i-1}$ is
  one-dimensional, there is a unique such VVS structure.
  Consequently, the VVS structure is determined by the filtration as
  follows.  Suppose $x,y \in V$ are non-zero.  Let $i, j$ be such that
  $x \in V_i \setminus V_{i-1}$ and $y \in V_j \setminus V_{j-1}$.  If
  $i < j$ then $\val(x) > \val(y)$.  If $i > j$ then $\val(x) <
  \val(y)$.  Finally, if $i = j$, the $\val(x) \le \val(y) \iff
  \val'(x) \le \val'(y)$, where $\val'$ is the unique VVS structure on
  the quotient space $V_i/V_{i-1}$.
\end{proof}
\begin{definition}
  If $\tau$ is a split VVS structure on $V$, the \emph{associated
  filtration} is the complete filtration $\{V_i\}_{0 \le i \le n}$ on
  $V$ from Proposition~\ref{filter}.
\end{definition}
\begin{remark} \label{filter2}
  Let $V$ be a finite-dimensional $K$-vector space and let $\tau$ be a
  split VVS structure on $V$ with splitting basis $v_1,\ldots,v_n$.
  Let $\tau'$ be another VVS structure on $V$.  Then $\tau' = \tau$ if
  and only if
  \begin{equation*}
    \Gamma + \val_{\tau'}(v_1) > \cdots > \Gamma + \val_{\tau'}(v_n),
  \end{equation*}
  i.e., $v_1,\ldots,v_n$ is a splitting basis of $\tau'$.  Indeed, if
  $v_1,\ldots,v_n$ is a splitting basis of $\tau'$, then $\tau'$ is
  split and has the same associated filtration as $\tau$, so $\tau' =
  \tau$ by Proposition~\ref{filter}.
\end{remark}

\begin{proposition} \label{split-code}
  Work in the $\Ldiv$-structure $K$.  If $\tau$ is a split definable VVS
  structure on $K^n$, then the code $\ulcorner \tau \urcorner$ is
  interdefinable with a real tuple in $K$.
\end{proposition}
\begin{proof}
  The code for $\tau$ is the code for the associated filtration
  $\{V_i\}_{0 \le i \le n}$.  Subspaces of $K^n$ have real codes by
    \cite[Lemma~4.3]{johnson}.
\end{proof}

\textbf{For the rest of the section, assume $K \models \pCF$.}
\begin{fact} \label{tcom} 
  Suppose $M \succeq K$, $b \in M^1$, and $\tp(b/K)$ is definable.  If
  $\val(b) \ge 0$, then there is $a \in K$ such that $b - a$ is
  $K$-infinitesimal, in the sense that $\val(b - a) > \Gamma(K)$.
\end{fact}
Fact~\ref{tcom} is well-known, and can be extracted from the proofs of
\cite[Lemma~2.22]{johnson-yao} or the $\pCF$ case of
\cite[Theorem~5.9]{cubides-ye}.
\begin{lemma} \label{split}
  Suppose $M \succeq K$, $a$ is a tuple in $M$, and $\tp(a/K)$ is definable.
  Let $V \subseteq K(a)$ be a finite-dimensional $K$-linear
  subspace, with the induced VVS structure.  Then $V$ is split.
\end{lemma}
\begin{proof}
  By definable Skolem functions, $\dcl(Ka) \preceq M$.  Shrinking
  $M$ we may assume $M = \dcl(Ka)$.  Then $\tp(b/K)$ is definable
  for any finite tuple $b$ in $M$.  Indeed, we can write $b$ as
  $f(a)$ for some $K$-definable function $f$, and then $\tp(b/K)$
  is the pushforward of the definable type $\tp(a/K)$ along the
  definable function $f$.  Let $n = \dim_K(V)$.
  \begin{claim}
    For $0 \le i \le n$, there is a basis $b_1,\ldots,b_n$ of $V$ such
    that \[\Gamma(K) + \val(b_1) < \Gamma(K) + \val(b_2) < \cdots < \Gamma(K) +
    \val(b_i),\] and $\Gamma(K) + \val(b_i) < \Gamma(K) + \val(b_j)$ for $j > i$.
  \end{claim}
  \begin{claimproof}
    Proceed by induction on $i$.  For the base case $i = 0$, any basis
    is suitable.  Suppose $c_1,\ldots,c_n$ is a suitable basis for
    some $i < n$.  Permuting $c_{i+1},\ldots,c_n$, we may assume
    $\val(c_1) < \cdots < \val(c_n)$.  For $j > i+1$, the model $M$
    contains $c_jc_{i+1}^{-1}$, and so $\tp(c_j c_{i+1}^{-1} / K)$ is
    definable.  Note $\val(c_j c_{i+1}^{-1}) = \val(c_j) -
    \val(c_{i+1}) \ge 0$.  By Fact~\ref{tcom}, there is $u_j \in K$ such that
    $c_j c_{i+1}^{-1} - u_j$ is $K$-infinitesimal, in the sense that
    \begin{equation*}
      \val(c_j c_{i+1}^{-1} - u_j) > \Gamma(K).
    \end{equation*}
    Then $\val(c_j - u_j c_{i+1}) > \Gamma(K) + \val(c_{i+1})$.
    Replacing $c_j$ with $c_j - u_j c_{i+1}$ for $j = i+2,\ldots,n$,
    we get a suitable basis for $i+1$.
  \end{claimproof}
  Taking $i = n$, we get a splitting basis for $V$.
\end{proof}
The proof of Lemma~\ref{split} is based on the idea of \cite[Lemma~5.11]{cubides-ye}.

\section{Analysis of definable types in $\pCF$} \label{analysis}

\subsection{Quantifier-free types}
Work in the $\Ldiv$ language, where $\pCF$ does \emph{not}
have quantifier elimination.  Let $K$ be a small model of $\pCF$,
embedded in a monster model $\Mm$.  Let $K[x_1,\ldots,x_n]_{< d}$ be
the set of polynomials with homogeneous degree less than $d$.  We abbreviate the tuple $(x_1,\ldots,x_n)$ as $x$.
\begin{lemma} \label{qf-an}
  Suppose $a \in \Mm^n$ and $q = \qftp(a/K)$.  Let $I_d$ be
  the kernel of the map
  \begin{align*}
    K[x]_{<d} &\to \Mm \\
    P(x) &\mapsto P(a).
  \end{align*}
  Let $\tau_d$ be the VVS structure on $K[x]_{<d}/I_d$ induced as a
  subspace of $\Mm$.
  \begin{enumerate}
  \item \label{qan1} The subspaces $I_d$ and VVS structures $\tau_d$ depend only $q
    = \qftp(a/K)$.
  \item \label{qan2} Conversely, $q$ is determined by the collection
    of all $I_d$ and $\tau_d$.
  \item \label{qan3} If $q$ is definable, then each $\tau_d$ is definable.
  \item \label{qan4} If $q$ is definable, each $\tau_d$ is split.
  \end{enumerate}
\end{lemma}

\begin{proof}
  Every atomic $\Ldiv(K)$-formula has the form
  $P(x) = Q(x)$ or $\val(P(x)) \le \val(Q(x))$ for some $d$
  and some $P(x), Q(x) \in K[x]_{< d}$.  Note
  \begin{align*}
    q(x) \vdash (P(x) = Q(x)) &\iff P - Q \in I_d \\
    q(x) \vdash (\val(P(x)) \le \val(Q(x))) & \iff
    \val_{\tau_d}(P) \le \val_{\tau_d}(Q) \\
    P \in I_d & \iff q(x) \vdash (P(x) = 0).
  \end{align*}
  Then (\ref{qan1})--(\ref{qan3}) are clear.  Part (\ref{qan4}) is a rephrasing of Lemma~\ref{split}.
\end{proof}
\begin{proposition} \label{qftp-code}
  Let $q$ be a definable quantifier-free $n$-type over $K$.  Then
  $\ulcorner q \urcorner$ is interdefinable with a real tuple
  (possibly infinite).
\end{proposition}
\begin{proof}
  Take $a \in \Mm^n$ realizing $q$.  Let $I_d$ and $\tau_d$ be
  as in Lemma~\ref{qf-an}.  The $I_d$ are definable (as
  $K$-linear subspaces of some $K^N$), and the VVS structures $\tau_d$ are definable
  by Lemma~\ref{qf-an}(\ref{qan3}).  By part (\ref{qan4}) of the lemma, each
  $\tau_d$ is split.
  By
  parts (\ref{qan1}) and (\ref{qan2}) of the lemma, $\ulcorner q \urcorner$ is
  interdefinable with the infinite tuple $(\ulcorner I_d \urcorner
  \ulcorner \tau_d \urcorner : d < \omega)$.  We can take $\ulcorner I_d \urcorner$ to be real by \cite[Lemma~4.3]{johnson} and $\ulcorner \tau_d\urcorner$ to be real by Proposition~\ref{split-code}.
\end{proof}

Recall from Remark~\ref{type-fam} that we identify partial types with families of definable sets, so that we can talk about a partial type \emph{extending} a family of definable sets, or a family of definable sets \emph{generating} a partial type.  Also recall the notion of \emph{$\Gamma$-families} from Definition~\ref{gfam}.

\begin{proposition} \label{qftp-gen}
  Let $q(x)$ be a definable quantifier-free $n$-type over $K$.  Then
  $q$ is generated by countably many $\Gamma$-families.  Moreover, we
  can take each $\Gamma$-family to be definable over $\ulcorner q
  \urcorner$.
\end{proposition}
In other words, there are $\ulcorner q \urcorner$-definable $\Gamma$-families
$\{X_{i,\gamma}\}_{\gamma}$ for $i < \omega$ such that $q(\Mm) =
\bigcap_i \bigcap_{\gamma \in \Gamma(K)} X_{i,\gamma}(\Mm)$.
\begin{proof}
  By Proposition~\ref{qftp-code} we can assume $\ulcorner q \urcorner$
  is a real tuple.  Let $K_0$ be $\dcl(\ulcorner q \urcorner)$; then
  $K_0 \preceq K$ by definable Skolem functions.
  
  Take $a \in \Mm^n$ realizing $q$.  Let $I_d$ and $\tau_d$ be
  as in Lemma~\ref{qf-an}.  By parts (\ref{qan3}) and (\ref{qan4}) of the
  lemma, each $\tau_d$ is split and definable.
  Take $Q_{d,1},\ldots,Q_{d,m_d} \in
  K[x]_{<d}$ a basis of $I_d$.  Take $P_{d,1},\ldots,P_{d,n_d} \in
  K[x]_{<d}$ such that $\{P_{d,1},\ldots,P_{d,n_d}\}$ is a splitting
  basis of $(K[x]_{<d}/I_d,\tau_d)$.  As $I_d$ and $\tau_d$ are
  $K_0$-definable and $K_0 \preceq K$, we can take $Q$ and $P$ to have
  coefficients in $K_0$.

  Suppose $a' \in \Mm^n$.  By parts (\ref{qan1}) and (\ref{qan2}) of
  Lemma~\ref{qf-an}, $\qftp(a'/K) = q$ if and only if
  $a'$ yields the same subspaces $I_d$ and the same VVS
  structures $\tau_d$ as $a$.  By Remark~\ref{filter2},
  $a'$ induces the VVS structure $\tau_d$ if and only if
  \begin{equation*}
    \Gamma(K) + \val(P_{d,1}(a')) > \Gamma(K) +
    \val(P_{d,2}(a')) > \cdots > \Gamma(K) +
    \val(P_{d,n_d}(a')).
  \end{equation*}
  In summary, $a'$ realizes $q = \qftp(a/K)$ if and only
  if the following conditions hold:
  \begin{itemize}
  \item $Q_{d,i}(a') = 0$ for any $d$ and $1 \le i \le m_d$.
  \item $P_{d,i}(a') \ne 0$ for any $d$ and $1 \le i \le n_d$.
  \item $\val(P_{d,i}(a')) - \val(P_{d,j}(a')) > \gamma$ for any $d$,
    any $1 \le i < j \le n_d$, and any $\gamma \in \Gamma(K)$.
  \end{itemize}
  Each of these conditions (for fixed $d, i, j$) is expressed by a
  $K_0$-definable $\Gamma$-family, possibly a constant family not depending on the parameter $\gamma \in \Gamma$.
\end{proof}

\subsection{Complete types} \label{s:def}
\begin{theorem} \label{tp-code}
  Let $K$ be a model of $\pCF$.  Let $q$ be a definable $n$-type over
  $K$.  Then $q$ is coded by a real tuple, possibly infinite.
\end{theorem}

\begin{proof}
  Let $\hat{q}$ be the quantifier-free part of $q$, and let $\ulcorner
  \hat{q} \urcorner$ be its code.  By Proposition~\ref{qftp-code}, we
  can take $\ulcorner \hat{q} \urcorner$ to be a real tuple.  Let $K_0
  = \dcl(\ulcorner \hat{q} \urcorner)$.  Then $K_0 \preceq K$ because
  $\pCF$ has definable Skolem functions.  The fact that $\hat{q}$ is
  $K_0$-definable implies that $q$ is $K_0$-definable, by
  Lemma~\ref{qf-def-2}(\ref{qd1}).  Then $\ulcorner q \urcorner \in
  \dcl(\ulcorner \hat{q} \urcorner)$.  On the other hand, $\hat{q}$ is
  determined by $q$, so $\ulcorner \hat{q} \urcorner \in
  \dcl(\ulcorner q \urcorner)$.  Therefore $\hat{q}$ and $q$ are
  interdefinable, and $q$ is coded by a real tuple because $\hat{q}$
  is.
\end{proof}
In particular, we see from the proof that any definable type is
interdefinable with its quantifier-free part.
\begin{theorem} \label{tp-gen}
  Let $K$ be a model of $\pCF$.  Let $q$ be a definable $n$-type over
  $K$.  Then $q$ is generated by countably many
  $\Gamma$-families definable over $\ulcorner q \urcorner$.
\end{theorem}
\begin{proof}
  Let $\hat{q}$ be the quantifier-free part of $q$.  As in the proof
  of Theorem~\ref{tp-code}, let $K_0 = \dcl(\ulcorner \hat{q}
  \urcorner) = \dcl(\ulcorner q \urcorner)$.  As the language is countable, the code $\ulcorner \hat{q} \urcorner$ is countable and $K_0$ is countable.
  By
  Lemma~\ref{qf-def-2}(\ref{qd2}), $q$ is generated by $\hat{q}$ and $q
  \restriction K_0$:
  \begin{equation*}
    \hat{q}(x) \cup (q \restriction K_0)(x) \vdash
    q(x).
  \end{equation*}
  By Proposition~\ref{qftp-gen}, $\hat{q}(x)$ is generated by countably many $\Gamma$-families.  On the other hand $q \restriction K_0$ consists of countably many $\Ldiv(K_0)$-formulas, each of which can be regarded as a degenerate, constant $\Gamma$-family.  Therefore $q(x)$ is generated by countably many $\Gamma$-families.
\end{proof}

Recall from Definition~\ref{refdef} that one definable family $\{D_a\}_{a \in X}$ \emph{refines} another family $\{D'_b\}_{b \in Y}$ if for any $b \in Y$, there is $a\in X$ such that $D_a \subseteq D'_b$.
\begin{proposition}\label{refine-g}
  Let $K$ be a model of $\pCF$.  Let $q \in S_n(K)$ be a definable
  type.  Let $\{D_a\}_{a \in Y}$ be a definable family of sets such
  that $q$ extends $\{D_a\}_{a \in Y}$.  Then there is a $\Gamma$-family $\{X_\gamma\}_{\gamma \in \Gamma}$ such that $q$ extends $\{X_\gamma\}_{\gamma \in \Gamma}$ and $\{X_\gamma\}_{\gamma \in \Gamma}$ refines $\{D_a\}_{a \in Y}$.
\end{proposition}
\begin{proof}
  Let $L \succeq K$ be a $\kappa$-saturated elementary extension for
  some $\kappa \gg \aleph_0$.  Let $r \in S_n(L)$ be the heir of
  $q$.  Note that $r(x) \vdash x \in D_a$ for any $a \in Y(L)$.
  By Theorem~\ref{tp-gen}, the $K$-definable type
  $r(x)$ is generated by $\{Z_{\alpha,\gamma} : \alpha <
  \omega, ~ \gamma \in \Gamma(L)\}$ where
  $\{Z_{\alpha,\gamma}\}_{\gamma \in \Gamma}$ is a $K$-definable
  $\Gamma$-family for each $\alpha < \omega$.  Then for each $a
  \in Y(L)$, $r(x) \vdash x \in D_a$, so by compactness there are
  $\alpha_1,\ldots,\alpha_m < \omega$ and
  $\gamma_1,\ldots,\gamma_m \in \Gamma(L)$ such that
  \begin{equation*}
    \bigcap_{i = 1}^m Z_{\alpha_i,\gamma_i} \subseteq D_a.
  \end{equation*}
  By saturation, we can assume the $\alpha_i$ always come from some
  finite set $S \subseteq \omega$.  Take $X_\gamma =
  \bigcap_{\alpha \in S} Z_{\alpha,\gamma}$.  Then
  $\{X_\gamma\}_{\gamma \in \Gamma}$ refines $\{D_a\}_{a \in Y}$.  The
  fact that $r(x) \vdash x \in X_\gamma$ for $\gamma \in
  \Gamma(L)$ implies that $q(x) = (r \restriction K)(x) \vdash x
  \in X_\gamma$ for $\gamma \in \Gamma(K)$.
\end{proof}

\subsection{Comparison with ACVF} \label{sec:acvf}
The proof of Theorem~\ref{tp-code} is formally similar to the analysis
of definable types in ACVF in \cite[\S5.2]{johnson}.  (ACVF is the
theory of algebraically closed valued fields.)  In both cases, one
finds a code for a definable type $\tp(b/K)$ by looking at
finite-dimensional subspaces of $K[b]$ as valued vector spaces.  In
both cases, one can analyze the structure of these valued vector
spaces, showing that they are coded in a natural way by tuples from
the home sort and geometric sorts (\cite[Theorem~5.3]{johnson},
Proposition~\ref{split-code}).  For $\pCF$, the geometric sorts turn
out to be unnecessary, and the resulting code is a tuple in the home
sort.  This contrasts with ACVF, where the geometric sorts are
strictly necessary.  In fact, in ACVF \emph{every} imaginary is
interalgebraic with the code of a definable type \cite[Theorems~4.1,
  5.14]{johnson}, so one could not hope for definable types to be
coded by tuples in the home sort alone.

\subsection{The necessity of infinite codes} \label{infinite}
Let $K$ be a $p$-adically closed field and $q \in S_n(K)$ be a
definable type.  By Theorem~\ref{tp-code}, the code $\ulcorner q
\urcorner$ can be taken to be a tuple in $K$, rather than $K^\eq$.  If $\dim(q) = 1$, then
we can even take a finite tuple \cite[Theorem~2.7]{jy-abelian}.  This
suggests the question of whether infinite tuples are necessary in
Theorem~\ref{tp-code}.
\begin{proposition} \label{long}
  There is a definable 2-type $q(x,y) \in S_2(\Qp)$ such that
  $\ulcorner q \urcorner$ is not interdefinable with any finite tuple
  in $\Qp$.
\end{proposition}
\begin{proof}
  By Fact~\ref{delon} below, \emph{any} type over $\Qp$ is definable.
  Take $a_0, a_1, a_2, \ldots \in \Qp$ algebraically independent over
  $\Qq$.  Take a monster model $\Mm \succeq \Qp$.  Take non-zero $b
  \in \Mm$ infinitesimal over $\Qp$, in the sense that $\val(b) >
  \Gamma(\Qp) = \Zz$.  Take $c \in \Mm$ a pseudolimit of the sequence
  $e_n = \sum_{i = 0}^n a_ib^{2i}$.  As noted above, $q(x,y) :=
  \tp(b,c/\Qp)$ is definable.  By Theorem~\ref{tp-code}, $q$ is coded
  by a tuple $\ulcorner q \urcorner$ in $\Qp$.  By
  definability of $q(x,y)$, the set
  \begin{equation*}
    \left\{ (\alpha_0,\ldots,\alpha_n) \in \Qp^{n+1} : \val\left(c -
    \sum_{i = 0}^n \alpha_i b^{2i}\right) > \val(b^{2n+1}) \right\}
  \end{equation*}
  is definable over $\ulcorner q \urcorner$ for each $n$.  By choice of $c$,
  \begin{equation*}
    \val\left(c - \sum_{i = 0}^n \alpha_i b^{2i}\right) > \val(b^{2n+1}) \iff
    (\alpha_0,\ldots,\alpha_n) = (a_0,\ldots,a_n).
  \end{equation*}
  Therefore $(a_0,\ldots,a_n) \in \dcl(\ulcorner q
  \urcorner)$ for each $n$.  Then
  \begin{equation*}
    n+1 = \dim(a_0,\ldots,a_n/\varnothing) \le \dim(\ulcorner
    q \urcorner / \varnothing)
  \end{equation*}
  for all $n$, which implies $\dim(\ulcorner q \urcorner/\varnothing)$
  is infinite, which implies $\ulcorner q \urcorner$ cannot be a
  finite tuple.
\end{proof}
\begin{fact}[{Delon~\cite{delon}}] \label{delon}
  Any type $p \in S_n(\Qp)$ is definable.
\end{fact}
We give a self-contained proof of Fact~\ref{delon} in the appendix (Theorem~\ref{delon2}).

\section{Application to \textit{dfg} groups} \label{group-app}
Recall the notion of groups with definable f-generics (\textit{dfg}) from Definition~\ref{dfg}.
\begin{theorem} \label{quot}
  In a model of $\pCF$, suppose $G$ is an interpretable \textit{dfg} group and $X$ is
  a definable set with a definable action of $G$.  Then the quotient
  space $X/G$ (the set of orbits) is definable.  More precisely, the
  interpretable set $X/G$ is in interpretable bijection with a
  definable set.
\end{theorem}
\begin{proof}
  Work in a monster model $\Mm$.  Fix a definable type $q \in G$ with boundedly
  many right translates.  Let $M_0$ be a small model over which $X, G,
  q$ are defined.  It suffices to show that each element of $X/G$ is
  interdefinable over $M_0$ with a real tuple (possibly infinite).  By
  Lemma~\ref{interdef} and Remark~\ref{interdef2}, it suffices to show that each element of
  $X/G$ is \emph{interalgebraic} over $M_0$ with a real tuple.  For $g
  \in G$, let $\rho_g : G \to G$ be the function $\rho_g(x) = x \cdot
  g$.  For $a \in X$, let $\tau_a : G \to X$ be the function
  $\tau_a(x) = x \cdot a$.
  
  Take any $e \in X/G$.  We claim that $e$ is interalgebraic over
  $M_0$ with a real tuple, possibly infinite.  Take any $a \in X$
  lifting $e \in X/G$.  Consider the global definable type
  $\tau_{a,\ast} q$ obtained by pushing forward $q$ along $\tau_a$.
  Then $\tau_{a,\ast} q$ is concentrated on the definable set $X$, so
  it is a definable $n$-type for some $n$ and $\ulcorner \tau_{a,\ast}
  q \urcorner$ is a real tuple by Theorem~\ref{tp-code}.  The
  range of $\tau_a$ is the orbit $G \cdot a$ coded by $e$, so the
  definable type $\tau_{a,\ast} q$ is concentrated on this orbit $G
  \cdot a$ and therefore $e \in \dcl^\eq(\ulcorner \tau_{a,\ast} q
  \urcorner)$.
  \begin{claim}
    If $\sigma \in \Aut(\Mm/eM_0)$, then $\sigma(\tau_{a,\ast} q)$ is
    $\tau_{a,\ast} \rho_{g,\ast} q$ for some $g \in G$.
  \end{claim}
  \begin{claimproof}
    As $\sigma$ fixes $e$, we have $\sigma(a) = g \cdot a$ for some $g
    \in G$.  As $q$ is $M_0$-definable, $\sigma(q) = q$, and so
    $\sigma(\tau_{a,\ast} q) = \tau_{\sigma(a),\ast} q$.  But
    $\tau_{\sigma(a)}(x) = x \cdot \sigma(a) = x \cdot g \cdot a =
    \tau_a(\rho_g(x))$.  Therefore,
    \begin{equation*}
      \sigma(\tau_{a,\ast} q) = \tau_{\sigma(a),\ast} q =
      \tau_{a,\ast} \rho_{g,\ast} q. \qedhere
    \end{equation*}
  \end{claimproof}
  By choice of $q$, there are only a small number of $\rho_{g,\ast} q$
  as $g$ varies, and therefore there are only a small number of
  $\sigma(\tau_{a,\ast} q)$ as $\sigma$ varies in $\Aut(\Mm/e)$.  It
  follows that $\ulcorner \tau_{a,\ast} q \urcorner \in \acl^\eq(e)$.
  Thus $e$ is interalgebraic with the real tuple $\ulcorner
  \tau_{a,\ast} q \urcorner$.
\end{proof}
The following corollary is immediate:
\begin{corollary}
  Let $G$ be a definable group and let $H$ be a definable normal
  subgroup.  If $H$ has \textit{dfg}, then $G/H$ is definable.
\end{corollary}
We can also strengthen some results of Pillay and Yao:
\begin{corollary} \label{dim0-quot}
  Let $G$ be a definable group and let $H$ be a definable subgroup
  with \textit{dfg}.  If $\dim(H) = \dim(G)$, then $H$ has finite index in $G$.
\end{corollary}
\begin{proof}
  The quotient space $G/H$ is a definable set, and $\dim(G/H) +
  \dim(H) = \dim(G)$, so $\dim(G/H) = 0$, which implies $G/H$ is
  finite.
\end{proof}
In particular, if $H$ is an open subgroup of $G$ (with respect to
  the Pillay topology on $G$), and $H$ has \textit{dfg}, then $H$ has finite index.  This answers Question 2 in \cite{pillay-yao}.

\section{Directed families and definable types in $\pCF$} \label{ss}

Let $K$ be a model of $\pCF$. 

\begin{fact} \label{fact:simon-star}
Let $\Mm \succeq K$ be a monster model. Let $\varphi(x,b)$ with $b\in \Mm$ be non-forking over $K$ (or equivalently by~\cite{kap-cher12}, non-dividing over $K$). Then $\varphi(x,b)$ extends to a global $K$-definable type.
\end{fact}
This was shown for a large class of dp-minimal theories, including those with definable Skolem functions, by Simon and Starchenko~\cite{simon_star_14}.
By a compactness argument Fact~\ref{fact:simon-star} can be restated, in a slightly weaker form, as follows. Recall that $S^\df_n(K)$ denotes the space of definable $n$-types over $K$.
\begin{corollary}\label{cor:sim-star-fact}
If $\mathcal{F}\subseteq \mathcal{P}(K^n)$ is a definable family of sets in $K$ with the finite intersection property, then $\mathcal{F}$ can be partitioned into finitely many subfamilies, each of which extends to a definable type in $S^{\df}_n(K)$.
\end{corollary}

\begin{proof}
Let $\Mm \succeq K$ be a monster model. Let $\varphi(x,y)$ and $\psi(y)$ be $\Ldiv(K)$-formulas such that $\mathcal{F}=\{\varphi(K,a) : a\in \psi(K)\}$ has the finite intersection property.  If $q \in S^{\df}_n(K)$, let $q^\Mm$ denote the extension of $q$ by the same definition scheme to a global $K$-definable type.  For each $q \in S^{\df}_n(K)$, let $D_q = \{a \in \psi(\Mm) : q^\Mm(x) \vdash \varphi(x,a)\}$.  Each $D_q$ is $K$-definable.  If $a \in \psi(\Mm)$, then $\varphi(x,a)$ doesn't divide over $K$ because $\mathcal{F}$ has the finite intersection property.  By Fact~\ref{fact:simon-star}, $\varphi(x,a) \in q^\Mm(x)$ for some $q \in S_n^{\df}(K)$.  Therefore $\psi(\Mm) \subseteq \bigcup_{q \in S^{\df}_n(K)} D_q$.  By compactness, there are finitely many $q_1,\ldots,q_m \in S_n^{\df}(K)$ with $\psi(\Mm) \subseteq \bigcup_{i = 1}^m D_{q_i}$. This means that for every $a \in \psi(\Mm)$, there is some $i$ such that $\varphi(x,a) \in q_i^{\Mm}(x)$.  Therefore, for every $a \in \psi(K)$, there is some $i$ such that $\varphi(x,a) \in q_i(x)$.  
\end{proof}

\begin{remark}\label{rem:pq}
A family of sets $\mathcal{S}$ has the \emph{$(m,n)$-property}, for integers $m\geq n\geq 1$, if the sets in $\mathcal{S}$ are non-empty and, for any $m$ distinct sets in $\mathcal{S}$, there exists $n$ among them with non-empty intersection. 

Let $\Mm \succeq K$ be a monster model. In Appendix B (Corollary~\ref{cor:pq-divide} and Fact~\ref{fct:vc}) we show that, for any $\Ldiv(K)$-formulas $\varphi(x,y)$ and $\psi(y)$, if the family $\mathcal{F}=\{\varphi(K,a) : a\in \psi(K)\}$ has the $(m,2|x|)$-property, for some $m\geq 2|x|$, then $\varphi(x,a)$ does not divide over $K$ for any $a\in \psi(\Mm)$. It follows that, by minimally adapting its proof, Corollary~\ref{cor:sim-star-fact} can be strengthened to the following statement:
\begin{quote}
    Let $\mathcal{F} \subseteq \mathcal{P}(K^n)$ be a definable family of sets in $K$ with the $(m,2n)$-property for some $m\geq 2n$. Then $\mathcal{F}$ can be partitioned into finitely many subfamilies, each of which extends to a definable type in $S^{\df}_n(K)$.
\end{quote}
\end{remark}
Say that a family of sets $\mathcal{F}$ is \emph{downward directed} if $\mathcal{F}$ is non-empty, and for any $X, Y \in \mathcal{F}$ there is $Z \in \mathcal{F}$ such that $Z \subseteq X \cap Y$.  
\begin{lemma}\label{lem:down-dir}
Let $\mathcal{F}$ be a definable downward directed family of sets in a structure $M$ and $p_1,\ldots,p_n$ be types over $M$ such that, for every $D\in \mathcal{F}$, there is some $i\leq n$ such that $D\in p_i$. Then there exists some $i\leq n$ such that $\mathcal{F}\subseteq p_i$.
\end{lemma}
\begin{proof}
Towards a contradiction suppose that, for every $i\leq n$, there exists some $D_i \in \mathcal{F}$ with $D_i \notin p_i$. By downward directedness there exists some $D\in \mathcal{F}$ with $D\subseteq \bigcap_{i\leq n} D_i$, which will satisfy $D\notin p_i$ for every $i\leq n$.   
\end{proof}

Recall the notion of ``$\Gamma$-family'' from Definition~\ref{gfam}.
\begin{proposition} \label{refine}
  Let $\mathcal{F}$ be a definable downward directed family of non-empty sets in $K$.
  \begin{enumerate}
  \item \label{ref1} $\mathcal{F}$ extends to a definable type.
  \item \label{ref2} $\mathcal{F}$ is refined by some $\Gamma$-family.
  \end{enumerate}
\end{proposition}
\begin{proof}
  \begin{enumerate}
  \item $\mathcal{F}$ has the finite intersection property.  Apply
    Corollary~\ref{cor:sim-star-fact} and Lemma~\ref{lem:down-dir}.
  \item By part (\ref{ref1}), there is a definable type $p$ refining
    $\mathcal{F}$.  Apply Proposition~\ref{refine-g}. \qedhere
  \end{enumerate}
\end{proof}

\begin{theorem} \label{tri}
  Let $\mathcal{F}$ be a definable family of sets in $K$.
  The following are equivalent:
  \begin{enumerate}
  \item \label{tr1} $\mathcal{F}$ extends to a definable type $p \in
    S^{\df}_n(K)$.
  \item \label{tr2} $\mathcal{F}$ is refined by a $\Gamma$-family.
  \item \label{tr3} $\mathcal{F}$ is refined by a downward directed definable
    family of non-empty sets.
  \end{enumerate}
\end{theorem}
\begin{proof}
  \begin{description}
  \item[$(\ref{tr1})\Rightarrow(\ref{tr2}).$] Proposition~\ref{refine-g}.
  \item[$(\ref{tr2})\Rightarrow(\ref{tr3}).$] $\Gamma$-families are downward directed.
  \item[$(\ref{tr3})\Rightarrow(\ref{tr1}).$] Proposition~\ref{refine}(\ref{ref1}). \qedhere
  \end{description}
\end{proof}

\section{Spaces of definable types in $\pCF$} \label{prodsec}

In this section, we review the notions of pro-definable and strict
pro-definable sets from \cite{kamensky} and \cite[Section~2.2]{HL}, as
well as the results of Cubides Kovacsics, Hils, and Ye \cite{cubides-ye,
  cubides-ye-hils} on the strict pro-definability of the space of
definable types in certain theories including $\pCF$.

Let $\Mm$ be a monster model of some theory.  Recall that a
\emph{$\ast$-type} is a partial type in infinitely (or finitely) many
variables.  A \emph{pro-definable set} is a set defined by a (small)
$\ast$-type\footnote{More precisely, these should be called
``$\ast$-definable sets'' rather than ``pro-definable sets.''
However, the categories of pro-definable sets and $\ast$-definable
sets are equivalent:  This is mentioned in \cite[Section~2.2]{HL},
and easy to see from the explicit description of the
pro-definable category in \cite[Corollary~8]{kamensky}.  We believe
$\ast$-definable sets are conceptually simpler than genuine
pro-definable sets.}.  For example, definable sets and type-definable
sets (in finitely many variables) are pro-definable.  Infinite
products of definable sets are pro-definable.  If $X, Y$ are
pro-definable sets, then a function $f : X \to Y$ is \emph{pro-definable} if
the graph of $f$ is pro-definable as a subset of $X \times Y$.  This
yields a category of pro-definable sets.  If $X$ is
pro-definable, a subset $D \subseteq X$ is \emph{relatively definable}
(in $X$) if $D$ is defined as a subset of $X$ by an $\mathcal{L}(\Mm)$-formula (mentioning only finitely many variables).  By compactness, this is
equivalent to $D$ and $X \setminus D$ both being pro-definable.

Any pro-definable set $X$ sits inside a product $X \subseteq \prod_{i
  \in I} D_i$ of definable sets, where $I$ is small, but possibly
infinite (for example, we can take $D_i$ to be the sort of the $i$th
variable in the tuple of variables).  For $J \subseteq I$, let $\pi_J : \prod_{i
  \in I} D_i \to \prod_{i \in J} D_i$ be the projection.  Following
Hrushovski and Loeser \cite[Section~2.2]{HL}, one says that $X$ is
\emph{strictly pro-definable} if the following equivalent conditions
hold:
\begin{enumerate}
\item \label{c1} For any finite $J \subseteq I$, the set $\pi_J(X)$ is a
  definable subset of $\prod_{i \in J} D_i$.
\item \label{c2} For any definable set $Y$ and pro-definable function $f : X \to
  Y$, the image $f(X)$ is definable.
\end{enumerate}

The first condition is more practical to work with, while the second
condition shows that strictness is an intrinsic property.  The
equivalence of (\ref{c1}) and (\ref{c2}) is left as an exercise to the
reader.
A \emph{(strictly) pro-interpretable} set is a (strictly)
pro-definable set in $\Mm^\eq$.  For future use, we record a
useful fact about strict pro-definability:
\begin{remark}[Quantification over strictly pro-definable sets] \label{quant}
  Let $X, Y$ be pro-definable sets.  Let $R \subseteq X \times Y$ be
  relatively definable.  If $X$ is strictly pro-definable, then the sets
  \begin{gather*}
    R_\exists = \{b \in Y : (\exists a \in X) R(a,b)\} \\
    R_\forall = \{b \in Y : (\forall a \in X) R(a,b)\} 
  \end{gather*}
  are relatively definable subsets of $Y$.
\end{remark}
\begin{proof}
  Let $\varphi(x,y)$ be an $\mathcal{L}(\Mm)$-formula defining $R$.  Then $\varphi(x,y)$
  uses only finitely many variables from $x$, so it is equivalent to
  $\varphi'(\pi(x),y)$ where $\pi$ is a coordinate projection onto
  finitely many coordinates.  Replacing $X$ with $\pi(X)$ and
  $\varphi$ with $\varphi'$, we may assume $X$ is definable, rather than pro-definable.  Then
  $R_\exists$ and $R_\forall$ are defined by the first-order $\mathcal{L}(\Mm)$-formulas
  $\exists x \in X~ \varphi(x,y)$ and $\forall x \in X ~
  \varphi(x,y)$.
\end{proof}

Recall that $S^{\df}_n(M)$ denotes the space of definable types over a model
$M$.  Recall that an $\mathcal{L}$-theory $T$ has \emph{uniform
definability of definable types\footnote{Called ``uniform definability
of types'' in \cite{cubides-ye}.}} (UDDT) if definable types are uniformly
definable in models of $T$: for any $\mathcal{L}$-formula
$\varphi(x,y)$ there is an $\mathcal{L}$-formula $\psi(y,z)$ such that
if $M \models T$ and $p \in S^{\df}_x(M)$, then there is
$c_{p,\varphi} \in M^z$ such that
\begin{equation*}
  \{b \in M^y : \varphi(x,b) \in p(x)\} = \psi(M,c_{p,\varphi}).
\end{equation*}

Replacing $\psi$ with an $\mathcal{L}^\eq$-formula, we may assume that
\begin{equation*}
  \psi(M,c) = \psi(M,c') \implies c = c',
\end{equation*}
and then $c_{p,\varphi}$ is uniquely determined.  Fix such a formula $\psi = d(\varphi)$ for each $\varphi$. For any variable tuple $x$, let $\mathcal{L}(x)$ denote the set of partitioned $\mathcal{L}$-formulas of the form $\varphi(x;y)$ (with varying $y$). Note that the map
\begin{equation*}
  p \mapsto (c_{p,\varphi} : \varphi \in \mathcal{L}(x)) \tag{$\dag$}
\end{equation*}
is injective on $S^{\df}_x(M)$.
Now, suppose the model $M$ is a monster model $\Mm$.  If $X \subseteq \Mm^x$ is definable,
let $X^{\df}$ be the image of $S^{\df}_X(\Mm)$ under the map ($\dag$).
\begin{fact} \label{pro-d}
  Assuming UDDT, $X^{\df}$ is pro-interpretable.
\end{fact}
This fact is proved by Cubides Kovacsics and Ye
\cite[Proposition~4.1]{cubides-ye} building off an argument of
Hrushovski and Loeser \cite[Lemma~2.4.1]{HL}.  Fact~\ref{pro-d} lets
us identify $S^{\df}_X(\Mm)$ with the pro-interpretable set $X^{\df}$.
Restricting to $\Aut(\Mm/M)$-invariant points, we can identify
$S^{\df}_X(M)$ with $X^{\df}(M)$ for any small model $M$.
\begin{fact}[{Cubides-Ye~\cite{cubides-ye}}] \label{uddt}
  The theory $\pCF$ has UDDT.
\end{fact}
Consequently, the pro-interpretable set $X^\df$ exists for any definable set $X$.
\begin{fact}[{Cubides, Hils, Ye~\cite{cubides-ye-hils}}] \label{strict}
  In $\pCF$, the pro-interpretable set $X^{\df}$ is strictly pro-interpretable.
\end{fact}
To be self-contained, we give proofs of Facts~\ref{uddt} and \ref{strict} in the appendix (Theorems~\ref{uddt2} and \ref{strict3}).

The existence of real codes for definable types implies the following:
\begin{theorem} \label{pro-d-literal}
  Work in a monster model $\Mm$ of $\pCF$.
  If $X$ is a definable set, then the pro-interpretable set $X^{\df}$
  is pro-definable.  (More precisely, it is in pro-interpretable
  bijection with a pro-definable set.)
\end{theorem}
\begin{proof}
  Fix a small set $A$ over which $X$ is definable.
  If $q \in S^{\df}_X(\Mm)$, then the tuple $c = (c_{q,\varphi} : \varphi
  \in \Ldiv(x)) \in X^{\df}$ is a code for $q$.  By
  Theorem~\ref{tp-code}, $c$ is interdefinable with a real tuple.  It
  follows that
  \begin{quote}
    Every point $c$ in $X^{\df}$ is interdefinable over $A$ with a real tuple.
  \end{quote}
  The following general fact completes the proof:
  \begin{claim}
    If $X$ is pro-interpretable over $A$ and every point of $X$ is interdefinable over $A$ with a real tuple, then there is a bijection $f : X \to Y$ where $Y$ is pro-definable over $A$ and $f$ is pro-interpretable over $A$.
  \end{claim}
  The proof of this fact is subtle, so we include it.
    It suffices to find an $A$-pro-interpretable injection $f : X \to \prod_{i \in I} D_i$ where the $D_i$ are $A$-definable, as we can then take $Y$ to be the $A$-pro-definable image of $f$.  Equivalently, we need to find a jointly injective family $\mathcal{F} = \{f_i\}_{i \in I}$ of $A$-pro-interpretable maps $f_i : X \to D_i$ where the sets $D_i$ are $A$-definable.  We may as well take $\mathcal{F}$ to be the collection of \emph{all} $A$-pro-interpretable maps $f : X \to D$ with $A$-definable $D$.
    
    Fix distinct $x, x' \in X$.  We must find $f : X \to D$ in $\mathcal{F}$ separating $x$ and $x'$.  If $x \not\equiv_A x'$, then there is a relatively $A$-interpretable set $U \subseteq X$ distinguishing $x$ and $x'$, and so we can find an $A$-pro-interpretable function $f : X \to \{0,1\}$ separating $x$ and $x'$.  Next suppose that $x \equiv_A x'$.  Take $\sigma \in \Aut(\Mm/A)$ with $\sigma(x) = x'$.  By assumption, there is a real tuple $y$ interdefinable with $x$ over $A$.  Then $\sigma(y) \ne y$.  Let $z$ be one of the coordinates of $y$ such that $\sigma(z) \ne z$.  Then $z \in \dcl^\eq(x)$, so we can write $z$ as $f_0(x_0)$ for some finite subtuple of $x_0$ and some $A$-interpretable function $f_0$.  Let $\pi$ be the coordinate projection sending $x$ to $x_0$.  Extending $f_0$ by zero, we may assume the domain of $f_0$ contains $\pi(X)$.  Let $f = f_0 \circ \pi$.  Then $f$ is an $A$-pro-interpretable function from $X$ to the home sort, and $f(x) = z$.  Finally, \[ f(x') = f(\sigma(x)) = \sigma(f(x)) = \sigma(z) \ne z = f(x),\]
    and so $f$ separates $x$ and $x'$ as desired.
\end{proof}

For future use, we make two observations on arbitrary theories with UDDT:

\begin{remark}[Assuming UDDT] \label{deffam}
  Let $\varphi(x,y)$ be an $\mathcal{L}$-formula and let $X$ be a definable set.
  Under the identification of $S^{\df}_X(\Mm)$ with $X^{\df}$, the set
  \begin{equation*}
    \{(b,p) \in \Mm^y \times X^{\df} : \varphi(\Mm,b) \in p\}
  \end{equation*}
  is relatively definable
  in $\Mm^y \times X^{\df}$, essentially by
  construction.\footnote{It's defined by $\psi(y,\pi(z))$, where
  $\psi$ is the uniform definition of $\varphi$ and $\pi$ is the
  coordinate projection on $X^{\df}$ picking out the coordinate
  $c_{p,\varphi}$ from the infinite tuple $(c_{p,\varphi} : \varphi
  \in \mathcal{L}(x))$.}  It follows that if $\{D_b\}_{b \in Y}$ is a
  definable family of subsets of $X$, then the following set is
  relatively definable in $Y \times X^{\df}$:
  \begin{equation*}
    \{(b,p) \in Y \times X^{\df} : D_b \in p\}.
  \end{equation*}
\end{remark}
\begin{lemma}[Assuming UDDT] \label{strange}
  Let $\{\mathcal{F}_i\}_{i \in I}$ be a small collection of definable
  families.  Suppose that for each finite $I_0 \subseteq I$, the
  definable family $\bigcup_{i \in I_0} \mathcal{F}_i$ extends to a
  definable type.  Then $\bigcup_{i \in I} \mathcal{F}_i$ extends to a
  definable type.
\end{lemma}
\begin{proof}
  For each $i \in I$, let $D_i \subseteq X^\df$ be the set of $p \in S^{\df}_n(\Mm)$
  extending $\mathcal{F}_i$.  We claim that $D_i$ is relatively
  definable in $X^{\df}$.  Indeed, if $\mathcal{F}_i$ is $\{Y_c\}_{c
    \in Z}$, then
  \begin{equation*}
    D_i = \{p \in X^{\df} : (\forall c \in Z) ~ Y_c \in p\},
  \end{equation*}
  and this condition is first-order by Remark~\ref{deffam}.

  The assumption is that any finite intersection of $D_i$'s is
  non-empty, and the conclusion is that the intersection of all
  $D_i$'s is non-empty.  As the $D_i$'s are relatively definable
  subsets of the pro-interpretable set $X^{\df}$, the conclusion holds by
  compactness.
\end{proof}

\section{Definable types in $\pCF^\eq$} \label{imaginary}

In this section we extend results from Sections~\ref{analysis}, \ref{ss} and \ref{prodsec} from $\pCF$ to $\pCF^\eq$. We use these in Section~\ref{topsec} to prove results about interpretable topological spaces. If one only cares about \emph{definable} topological spaces, this section can be skipped.

\begin{lemma}[{like Fact~\ref{delon}}] \label{delon-lift}
  Every type over $\Qp^\eq$ is definable.
\end{lemma}
\begin{proof}
  If $f : \tilde{X} \to X$ is an interpretable surjection with $\tX$
  definable, then we can lift types in $X$ to types in $\tilde{X}$:
  the pushforward map $S_{\tX}(\Qp) \to S_X(\Qp)$ is surjective.  Every type in $S_{\tX}(\Qp)$ is definable by Fact~\ref{delon}, and definable types push forward to definable types.
\end{proof}

\subsection{Lifting definable types from $K^\eq$ to $K$}

Work in a model $K \models \pCF$.  Fix an interpretable surjection $\pi : \tX \to X$ where $\tX$ is definable.  Let $S^{\df}_X(K)$ and
$S^{\df}_{\tX}(K)$ be the spaces of $K$-definable types on $X$
and $\tX$, respectively, and let $\pi_* : S^{\df}_{\tX}(K) \to
S^{\df}_X(K)$ be the pushforward map.

\begin{lemma} \label{s0}
  Let $\mathcal{F}$ be an interpretable family of subsets of $X$ with
  the finite intersection property.
  Then there are finitely
  many definable types $q_1, \ldots, q_n \in S^{\df}_{\tX}(K)$ such
  that each set $D \in \mathcal{F}$ belongs to at least one $\pi_*
  q_i$.
\end{lemma}
\begin{proof}
  Note that $D \in \pi_* q_i \iff \pi^{-1}(D) \in q_i$.  Apply
  Corollary~\ref{cor:sim-star-fact} to the family $\{\pi^{-1}(D) : D
  \in \mathcal{F}\}$.
\end{proof}

The proof of Lemma~\ref{s1} below resembles the proof of Lemma~\ref{lem:down-dir}, except we rely on distal cell decomposition in place of downward directedness.

\begin{lemma} \label{s1}
  Let $\mathcal{F}$ be an interpretable family of subsets of $X$.  If
  $\mathcal{F}$ extends to a definable type, then $\mathcal{F}$
  extends to a definable type of the form $\pi_* q$ for some $q \in
  S^{\df}_{\tX}(K)$.
\end{lemma}
\begin{proof}
  Fix a type $r \in S^{\df}_{X}(K)$ extending $\mathcal{F}$.  By
  distal cell decomposition \cite[Fact~2.5]{reglem}, there is an
  interpretable family $\mathcal{G}$ such that any finite intersection
  of sets in $\mathcal{F}$ is a finite union of sets in $\mathcal{G}$.
  Let $\mathcal{G} \cap r$ be the set of $D \in \mathcal{G}$ with $D
  \in r$.  Then $\mathcal{G} \cap r$ is an interpretable family with
  the finite intersection property.  By Lemma~\ref{s0}, there are finitely many $q_1, \ldots, q_n
  \in S^{\df}_{\tX}(K)$ such that each set in $\mathcal{G} \cap r$
  is in $\pi_* q_i$ for at least one $i$.

  We claim that some $\pi_* q_i$ extends $\mathcal{F}$.  Otherwise,
  for $1 \le i \le n$ take $D_i \in \mathcal{F}$ with $D_i \notin
  \pi_* q_i$.  Each $D_i$ is in $r$, and $r$ is a complete type, so
  the intersection $\bigcap_{i = 1}^n D_i$ is in $r$.  We can write
  this intersection as $\bigcup_{j = 1}^m D'_j$ with $D'_j \in
  \mathcal{G}$.  As $r$ is a complete type, some $D'_j$ is in $r$, and
  therefore in $\mathcal{G} \cap r$.  By choice of the $q_i$, there is
  some $i$ such that $D'_j \in \pi_* q_i$.  But $D'_j \subseteq D_i$,
  so then $D_i \in \pi_* q_i$, contradicting the choice of $D_i$.
\end{proof}

\begin{lemma} \label{s2}
  If $K$ is sufficiently saturated, then the pushforward map $\pi_* : S^{\df}_{\tX}(K) \to S^{\df}_X(K)$
  is surjective.
\end{lemma}
\begin{proof}
  Fix $r \in S^{\df}_X(K)$.  As $r$ is definable, it is generated by
  a small union $\bigcup_{i \in I} \mathcal{F}_i$ of interpretable
  families $\mathcal{F}_i$.  For each $i \in I$, let
  $\tilde{\mathcal{F}}_i$ be the definable family $\{\pi^{-1}(D) : D
  \in \mathcal{F}_i\}$.  By Lemma~\ref{s1}, for any finite $I_0
  \subseteq I$ there is $q \in S^{\df}_{\tX}(K)$ such that
  $\bigcup_{i \in I_0} \mathcal{F}_i \subseteq \pi_* q$, or
  equivalently, $\bigcup_{i \in I_0} \tilde{\mathcal{F}}_i \subseteq
  q$.  By Lemma~\ref{strange}, there is $q \in S^{\df}_{\tX}(K)$
  with $\bigcup_{i \in I} \tilde{\mathcal{F}}_i \subseteq q$, or
  equivalently, $\bigcup_{i \in I} \mathcal{F}_i \subseteq \pi_*q$.
  Then $\pi_* q$ must be $r$.
\end{proof}

\subsection{Generation by $\Gamma$-families}
\begin{proposition}[{like Proposition~\ref{refine-g}}] \label{refine-g3}
  Suppose $K \models \pCF$ and $q$ is a definable type in $K^\eq$.
  Let $\mathcal{F}$ be an interpretable family such that $q$ extends
  $\mathcal{F}$.  Then there is an interpretable $\Gamma$-family
  $\mathcal{G}$ such that $q$ extends $\mathcal{G}$ and
  $\mathcal{G}$ refines $\mathcal{F}$.
\end{proposition}
\begin{proof}
  Passing to an elementary extension $K' \succeq K$, we may assume $K$
  is $\aleph_1$-saturated.\footnote{To return from $K'$ to $K$, let
  $b_0$ be a tuple in $K'$ defining the $\Gamma$-family $\mathcal{G}$,
  and write $\mathcal{G}$ as $\mathcal{G}_{b_0}$ to make the
  dependence on $b_0$ explicit.  Let $D$ be the set of $b \in K'$ such
  that $\mathcal{G}_b$ is a $\Gamma$-family refining $\mathcal{F}$ and
  extended by $q$.  As $q$ is $K$-definable, the set $D$ is
  $K$-definable.  As $K \preceq K'$, we can replace $b_0 \in D$ with a
  $K$-definable tuple $b_1 \in D$, getting a $K$-definable
  $\Gamma$-family.} Let $X$ be the interpretable set where $q$ lives,
  and let $\tX$ be a definable set with an interpretable surjection
  $\pi : \tX \to X$.  By $\aleph_1$-saturation and Lemma~\ref{s2}, there is $\tilde{q} \in S_{\tX}^{\df}(K)$ such
  that $q = \pi_* \tilde{q}$.  Let $\tilde{\mathcal{F}} =
  \{\pi^{-1}(D) : D \in \mathcal{F}\}$. Then $\tilde{\mathcal{F}}$ is
  a definable family extended by the definable type $\tilde{q}$.  By
  Proposition~\ref{refine-g}, there is a definable $\Gamma$-family
  $\mathcal{G}_0$ such that $\mathcal{G}_0$ refines
  $\tilde{\mathcal{F}}$, and $\tilde{q}$ extends $\mathcal{G}_0$.  Let
  \begin{equation*}
    \mathcal{G} = \{\pi(D) : D \in \mathcal{G}_0\}.
  \end{equation*}
  Then $\mathcal{G}$ is an interpretable $\Gamma$-family in $X$.  If
  $D \in \mathcal{G}_0$, then $D \in \tilde{q}$ (because $\tilde{q}$
  extends $\mathcal{G}_0$), so $\pi^{-1}(\pi(D)) \in \tilde{q}$
  (because $\pi^{-1}(\pi(D)) \supseteq D$), and then $\pi(D) \in \pi_*
  \tilde{q} = q$.  Therefore $q$ extends $\mathcal{G}$.  It remains to
  show that $\mathcal{G}$ refines $\mathcal{F}$.  Take some $D \in
  \mathcal{F}$.  Then $\pi^{-1}(D) \in \tilde{\mathcal{F}}$, so there
  is $E \in \mathcal{G}_0$ with $E \subseteq \pi^{-1}(D)$ because
  $\mathcal{G}_0$ refines $\tilde{\mathcal{F}}$. Then $\pi(E) \subseteq D$ and $\pi(E) \in \mathcal{G}$.
  This shows that $\mathcal{G}$ refines $\mathcal{F}$.
\end{proof}
This yields an analogue of Theorem~\ref{tp-gen}.
\begin{theorem} \label{tp-gen1}
  Suppose $K \models \pCF$ and $q$ is a definable type in $K^\eq$.
  Then $q$ is generated by a union of countably many $\Gamma$-families.
\end{theorem}
\begin{proof}
  As a definable type, $q$ is generated by a union of countably many
  interpretable families $\bigcup_{i = 1}^\infty \mathcal{F}_i$.
  (Indeed, $q$ \emph{is} a union of countably many interpretable
  families, one for each formula in the language.)  For each $i$, take
  an interpretable $\Gamma$-family $\mathcal{G}_i$ refining
  $\mathcal{F}_i$ and extended by $q$.  Then $\bigcup_{i = 1}^\infty
  \mathcal{G}_i$ generates $q$.
\end{proof}

\subsection{Directed families and definable types}

Let $K$ be a model of $\pCF$.

\begin{lemma}[{like Corollary~\ref{cor:sim-star-fact}}] \label{ssf2}
  If $\mathcal{F}$ is an interpretable family of sets in $K^\eq$ with
  the finite intersection property, then $\mathcal{F}$ can be
  partitioned into finitely many subfamilies, each of which extends to
  a definable type.
\end{lemma}
\begin{proof}
Let $X$ be the sort where the sets in $\mathcal{F}$ live, let $\tX$ be a 0-definable set with a 0-interpretable surjection to $X$, and apply Lemma~\ref{s0}.
\end{proof}

\begin{proposition}[{like Proposition~\ref{refine}}] \label{refine-i}
  Let $\mathcal{F}$ be an interpretable downward directed family of non-empty sets in $K$.
  \begin{enumerate}
  \item $\mathcal{F}$ extends to a definable type.
  \item \label{refi2} $\mathcal{F}$ is refined by some $\Gamma$-family of non-empty
    sets.
  \end{enumerate}  
\end{proposition}
\begin{proof}
  Like the proof of Proposition~\ref{refine}, using Lemma~\ref{ssf2}
  instead of Corollary~\ref{cor:sim-star-fact}, and
  Proposition~\ref{refine-g3} instead of Proposition~\ref{refine-g}.
\end{proof}

\begin{theorem}[{like Theorem~\ref{tri}}] \label{tri1}
  Let $\mathcal{F}$ be an interpretable family of sets in $K$.
  The following are equivalent:
  \begin{enumerate}
  \item $\mathcal{F}$ extends to a definable type.
  \item $\mathcal{F}$ is refined by a $\Gamma$-family.
  \item $\mathcal{F}$ is refined by a downward directed interpretable
    family of non-empty sets.
  \end{enumerate}
\end{theorem}
\begin{proof}
  Like the proof of Theorem~\ref{tri}, using
  Proposition~\ref{refine-g3} instead of Proposition~\ref{refine-g} and
  Proposition~\ref{refine-i} instead of Proposition~\ref{refine}.
\end{proof}

\subsection{Strict pro-definability in $\pCF^\eq$}
Work in a monster model $\Mm \models \pCF$.

\begin{theorem} \label{uddt-i}
  $\pCF^\eq$ has uniform definability of definable types (UDDT).
\end{theorem}
\begin{proof}
  (Compare with \cite[Remark~2.4.10]{cubides-ye-hils}.)
  Let $X$ be any product of
  sorts in $\Mm^\eq$.  Let $\varphi(x,y)$ be an
  $\Ldiv^\eq$-formula where $x$ lives in the sort $X$.  We must
  bound the complexity of the sets
  \begin{equation*}
    \{b \in \Mm^y : \varphi(x,b) \in r(x)\}
  \end{equation*}
  as $r$ ranges over $S^{\df}_X(\Mm)$.  By Lemma~\ref{s2}, this is the
  same as bounding the complexity of
  \begin{equation*}
    \{b \in \Mm^y : \varphi(x,b) \in (\pi_* q)(x)\}
  \end{equation*}
  as $q$ ranges over $S^{\df}_{\tX}(\Mm)$, where $\pi : \tX \to X$ is a $0$-interpretable surjection from a definable set $\tX$.  But
  \begin{equation*}
    \{b \in \Mm^y :\varphi(x,b) \in (\pi_* q)(x)\} = \{b \in \Mm^y :
    \varphi(\pi(w),b) \in q(w)\},
  \end{equation*}
  so this follows by UDDT in $\pCF$ applied to the formula
  $\varphi(\pi(w),y)$.
\end{proof}
By UDDT and Fact~\ref{pro-d}, there is a pro-interpretable set
$X^{\df}$ associated to any interpretable set $X$.
\begin{theorem} \label{strict2}
  If $X$ is an interpretable set, then $X^{\df}$ is strictly
  pro-interpretable.
\end{theorem}
\begin{proof}
  Take a definable set $\tX$ and an interpretable surjection $\pi : \tX
  \to X$.  By Lemma~\ref{s2}, the map $\tX^{\df} \to X^{\df}$ is
  surjective.  The strict pro-definability of $\tX^{\df}$
  (Fact~\ref{strict}) then implies strict pro-definability for
  $X^{\df}$.
\end{proof}
\begin{theorem} \label{lifts}
  Let $f : X \to Y$ be an interpretable surjection.  Then $f_* : X^{\df} \to
  Y^{\df}$ is surjective.
\end{theorem}
\begin{proof}
  Take a definable set $W$ and interpretable surjection $g : W \to X$.
  Applying Lemma~\ref{s2} to the interpretable surjection $f \circ g :
  W \to Y$, we see that $f_* \circ g_* : W^{\df} \to Y^{\df}$ is
  surjective, which implies $f_* : X^{\df} \to Y^{\df}$ is surjective.
\end{proof}
\begin{remark}
  The analogue of Theorem~\ref{lifts} for \emph{definable}
  surjections is trivial, because any definable surjection $f : X \to
  Y$ has a section $g : Y \to X$ by definable Skolem functions, and
  then $g_* : Y^{\df} \to X^{\df}$ witnesses that $f_* : X^{\df} \to
  Y^{\df}$ is surjective.
\end{remark}

\begin{question}\label{open-q}
  If the surjection $f : X \to Y$ of Theorem~\ref{lifts} is
  interpretable over a small model $M$, is the induced map $X^{\df}(M)
  \to Y^{\df}(M)$ surjective?  Equivalently, does every $M$-definable
  type in $Y$ lift to an $M$-definable type in $X$?  This doesn't
  follow formally from Theorem~\ref{lifts}, as there is a
  surjection of strictly pro-interpretable sets $f : X \to Y$ over
  $\Qp$ such that $X(\Qp) \to Y(\Qp)$ is not surjective.
  Specifically, take $X$ to be the value group and $Y$ to be the set of
  tuples $(x_0,x_1,x_2,\ldots) \in \{0,1\}^\omega$ satisfying $x_0 \le
  x_1 \le x_2 \ldots$.  Let $f : \Gamma \to Y$ be the map sending
  $\gamma$ to $(x_0,x_1,x_2,\ldots)$, where $x_i$ is $0$ if $i <
  \gamma$ and $1$ if $i \ge \gamma$.  Over the monster model, $f$ is
  surjective, but no $\gamma \in \Gamma(\Qp) = \Zz$ maps to
  $(0,0,0,\ldots) \in Y(\Qp)$.  On the other hand, when $M = \Qp$ the specific map $X^{\df}(M) \to Y^{\df}(M)$ \emph{is} surjective, because Lemma~\ref{delon-lift} ensures all types are definable.
\end{question}

\section{Definable and interpretable topological spaces} \label{topsec}

Throughout let $K\models \pCF$.

\begin{definition}
  If $Z \subseteq K^n$ is a definable set, a \emph{definable topology}
  on $Z$ is a topology $\tau$ with a definable basis of open sets.  That is, there is a definable family $\mathcal{B}$ of subsets of $Z$ that is a basis of opens for $\tau$.  A
  \emph{definable topological space} is a definable set with a
  definable topology.  An \emph{interpretable topological space} is a
  definable topological space in $K^\eq$.
\end{definition}

An analogue of the following result was proved in o-minimal expansions of ordered fields in~\cite[Corollary 40 (2)]{atw1}.

\begin{proposition}[Definable first countability]\label{prop:pCF_1st_countability}
Let $(Z,\tau)$ be an interpretable topological space in $K$. For any $z\in Z$, there exists a basis of open neighborhoods of $z$ given by a $\Gamma$-family (a $\Gamma$-basis).
\end{proposition}
\begin{proof}
Let $\mathcal{B}=\{B_{y} : y\in Y\}$ be an interpretable basis for $\tau$ and fix $z\in Z$. Let $Y_z=\{ y\in Y : z\in B_{y}\}$. Note that $\{B_{y} : y\in Y_z\}$ is an interpretable
basis of open neighborhoods of $z$. 

For any $y\in Y_z$, let
\[
Y(y)=\{ x\in Y : z\in B_{x} \subseteq B_{y}\}.
\]
The family $\{Y(y) : y\in Y_z\}$ of non-empty sets is interpretable and downward directed. We apply Proposition~\ref{refine}(\ref{ref2}) or Proposition~\ref{refine-i}(\ref{refi2}) to find a $\Gamma$-family $\{X_\gamma\}_{\gamma\in\Gamma}$ that refines it.

For any $\gamma\in \Gamma$, let 
\[
A_\gamma=\bigcup_{y\in X_\gamma} B_y. 
\]
Then $\{A_\gamma: \gamma\in \Gamma\}$ is a $\Gamma$-basis of open neighborhoods of $z$.
\end{proof}
Note that Proposition~\ref{prop:pCF_1st_countability} shows in particular that any interpretable topological space in $\Qp$ is first countable (i.e. every point has a countable basis of neighborhoods). 

Given a topological space $(Z,\tau)$ and a subset $Y \subseteq Z$, we let $\cl(Y)$ denote the closure of $Y$. We regard $K$ as a definable topological space via the usual valuation topology.

\begin{definition} \label{curve}
  Let $(Z,\tau)$ be an interpretable topological space in $K$.  An
  \emph{interpretable curve on $Z$} is an interpretable map $f : D \to Z$,
  where $D \subseteq K$ is definable with $0 \in \cl(D) \setminus D$, where $\cl(D)$ denotes the closure with respect to the valuation topology.
  \emph{We do not assume $f$ is continuous.} We say that $f$
  \emph{converges to $z \in Z$} if $\lim_{x \to 0} f(x) = z$, in the
  sense that for any neighborhood $U$ of $z$ there exists
  $\gamma_0\in \Gamma$ such that, for any $x \in D$ with $\val(x) \ge
  \gamma_0$, $f(x) \in U$.
\end{definition} 

\begin{remark}
  The theory $\pCF$ has definable Skolem functions.  Therefore, any definable surjection $f : X \to Y$ has a definable section, that is, a definable map $s : Y \to X$ with $f \circ s = \id_Y$.  The same does not hold in $\pCF^\eq$.  In other words, interpretable surjections need not have interpretable sections.  For example the valuation map $\val : K^\times \to \Gamma$ has no interpretable section.  However, if $f : X \to Y$ is an interpretable surjection and $Y$ is \emph{definable}, then $f$ has an interpretable section $s : Y \to X$.  To see this, take an interpretable surjection $\pi : \tilde{X} \to X$ with $\tilde{X}$ definable, and apply definable Skolem functions to the definable surjection $f \circ \pi : \tilde{X} \to Y$ to get a section $s_0 : Y \to \tilde{X}$. Then $f \circ \pi \circ s_0 = \id_Y$, so $\pi \circ s_0$ is an interpretable section of the original map $f$.  This argument is formally related to the proof of Lemma~\ref{interdef}(\ref{id1}).
  
  In what follows, when we apply ``definable Skolem functions'' to choose a function $s : Y \to X$, the domain $Y$ will always be definable. 
\end{remark}
Definable first countability (Proposition~\ref{prop:pCF_1st_countability}) and definable Skolem functions allow us to prove a form of definable curve selection.

\begin{lemma}[Definable curve selection]\label{curve-selection}
Let $(Z,\tau)$ be an interpretable topological space in $K$. For any interpretable set $Y\subseteq Z$ and any $z\in \cl(Y)$ there exists an interpretable curve in $Y$ converging to $z$.
\end{lemma}
\begin{proof}
Let $z \in \cl(Y)$ and, by Lemma~\ref{prop:pCF_1st_countability}, let $\{A_\gamma : \gamma\in \Gamma\}$ be a $\Gamma$-basis of neighborhoods of $z$. By definable Skolem functions let $f:K\setminus \{0\} \rightarrow Y$ be an interpretable map satisfying $f(x)\in A_{\val(x)} \cap Y$ for every $x\in K\setminus\{0\}$. Then clearly $f$ is an interpretable curve in $Y$ converging to $z$.
\end{proof}

Definable curve selection can be used to characterize continuity of interpretable functions in terms of whether limits of interpretable curves are maintained. We make this explicit in the following proposition, whose proof we omit, pointing the reader to~\cite[Chapter 6, Lemma 4.2]{dries-book} and~\cite[Proposition 42]{atw1} for analogous results in the o-minimal setting.

\begin{proposition}
Let $(Z,\tau)$ and $(Y,\mu)$ be interpretable topological spaces in $K$. Let $h:(Z,\tau)\rightarrow (Y,\mu)$ an interpretable map. Then, for any $z\in Z$, the map $h$ is continuous at $z$ if and only if, for every interpretable curve $f$ in $Z$, if $f$ converges to $z$ then the curve $h\circ f$ in $Y$ converges to $f(z)$.
\end{proposition}

Recall that if $\{x_i\}_{i < \omega}$ is a sequence in a topological space  $X$, then a point $p \in X$ is a \emph{cluster point} if every neighborhood of $p$ contains infinitely many terms in the sequence.  If $X$ is first-countable, then $p$ is a cluster point if and only if some subsequence of $\{x_i\}_{i < \omega}$ converges to $p$.  (This need not hold in a general topological space.)  The
analogous equivalence holds in our setting, thanks to definable first
countability:
\begin{proposition} \label{prop:cluster-point}
  Let $f : D \to Z$ be an interpretable curve in an interpretable
  topological space $(Z,\tau)$ in $K$.  Let $z \in Z$ be a point.  The following are equivalent:
  \begin{enumerate}
  \item \label{cp1} $z$ is a cluster point of $f$, in the sense that, for every $\gamma$,
     \begin{equation*}
      z \in \cl\{ f(x) : x \in D, \, \val(x) \ge \gamma\}.
    \end{equation*}
  \item \label{cp2} There is an interpretable curve $g$ that is a restriction of $f$ and converges to $z$.
  \end{enumerate}
\end{proposition}
\begin{proof}
  $(\ref{cp2})\Rightarrow(\ref{cp1})$ is easy.  We prove $(\ref{cp1})\Rightarrow(\ref{cp2})$.  Using Lemma~\ref{prop:pCF_1st_countability}, let $\{A_\gamma\}_{\gamma\in \Gamma}$ be a $\Gamma$-basis of open neighborhoods of $z$.  

For each $\gamma\in \Gamma$, set
\[
 J_{\gamma} := \{x \in D : \val(x) \ge \gamma, ~ f(x) \in A_{\gamma}\}.
\]
Observe that $\{J_\gamma\}_{\gamma\in \Gamma}$ is a $\Gamma$-family.

For each $\gamma\in \Gamma$, let $J_\gamma^{\min}$ denote the set of $x\in J_\gamma$ of minimum valuation. (Any non-empty bounded-below interpretable subset of $\Gamma$ has a minimum, because this holds in the standard model $\Qp$).
Set
\[
C:=\bigcup_{\gamma\in\Gamma} J_\gamma^{\min} \subseteq D \subseteq K. 
\]
By definition of the sets $J_\gamma$ it clearly holds that $0\in \cl(C)$, where the closure $\cl(C)$ is with respect to the valuation topology on $K$.  For any $\gamma$, let $\mu(\gamma)$ be the valuation of points in $J_\gamma^{\min}$, or equivalently, the minimum valuation of points in $J_\gamma$.  Note that
\[
\gamma \ge \gamma' \implies J_{\gamma} \subseteq J_{\gamma'} \implies \mu(\gamma) \ge \mu(\gamma'). \tag{$\ast$}
\]
We show that $f|_C$ converges to $z$.  Fix a basic neighborhood $A_\gamma \ni z$.  We claim that if $x \in C$ and $\val(x) > \mu(\gamma)$ then $f(x) \in A_\gamma$.  Fix $\gamma'$ such that $x \in J_{\gamma'}^{\min}$.  Then $\mu(\gamma) < \val(x) = \mu(\gamma')$.  By ($\ast$), $\gamma < \gamma'$.  By definition of $J_{\gamma'}$, $f(x) \in A_{\gamma'} \subseteq A_\gamma$.  This proves the claim.
\end{proof}

\begin{definition} \label{spec}
Fix an interpretable topological space $(Z,\tau)$ in $K$.  If $q \in S_Z(K)$ and $z \in Z$, then $q$ \emph{specializes} to $z$ (equivalently $z$ is a \emph{specialization} of $q$) if the following equivalent conditions hold:
\begin{enumerate}
    \item \label{sp1} For every interpretable open set $U \ni z$, we have $U \in q$.
    \item \label{sp2} For every interpretable closed set $C\in q$, we have $z \in C$.
    \item For every interpretable set $Y\in q$, we have $z \in \cl(Y)$.
\end{enumerate}
\end{definition}
\begin{remark} \label{basic-closed}
  Fix an interpretable basis of open sets on $(Z,\tau)$.
  In conditions (\ref{sp1}) and (\ref{sp2}) of Definition~\ref{spec}, it suffices to consider \emph{basic} open sets and \emph{basic} closed sets, respectively.  Here, a ``basic closed set'' is the complement of a basic open set.  Note that when $q$ is a definable type, the family of basic closed sets (or basic open sets) in $q$ is an interpretable family.
\end{remark}

\begin{definition}
An interpretable topological space $(Z,\tau)$ in $K$ is 
\begin{enumerate}
    \item \emph{directed-compact} if every downward directed interpretable family of non-empty closed subsets of $Z$ has non-empty intersection.
    \item \emph{curve-compact} if for every interpretable curve $f$ in $Z$ there exists another interpretable curve $g$ that is a restriction of $f$ and that converges.  By Proposition~\ref{prop:cluster-point}, we could equivalently say: every interpretable curve $f$ in $Z$ has a cluster point.
    \item \emph{type-compact} if every definable type concentrating on $Z$ specializes to a point in $Z$.
\end{enumerate}
\end{definition}

\begin{theorem}\label{thm:compactness}
Let $(Z,\tau)$ be an interpretable topological space in $K$. The following are equivalent. 
\begin{enumerate}
    \item \label{tc1} $(Z,\tau)$ is directed-compact.
    \item \label{tc1.5} Any $\Gamma$-family of closed sets in $Z$ has non-empty intersection.
    \item \label{tc2} $(Z,\tau)$ is curve-compact. 
    \item \label{tc3} $(Z,\tau)$ is type-compact. 
    \item \label{tc4} Any interpretable family of closed sets $\{C_y\}_{y\in Y}$ in $Z$ with the finite intersection property has a finite transversal, i.e. there exists a finite set $T$ with $T\cap C_y\neq \varnothing$ for every $y\in Y$.
\end{enumerate}
\end{theorem}

\begin{proof}

  \begin{description}
  \item[$(\ref{tc1})\Rightarrow(\ref{tc1.5}).$] $\Gamma$-families are downward directed.
  \item[$(\ref{tc1.5})\Rightarrow(\ref{tc2}).$] Let $f:D\subseteq K\rightarrow X$ be an interpretable curve. We apply (\ref{tc1.5}) to the $\Gamma$-family of non-empty sets
\[
\{f(x) : x\in D,\, \val(x)\geq \gamma\} \text{ for } \gamma\in \Gamma 
\]
to find a point in the $\tau$-closure of all of them (a cluster point of $f$). Then apply Proposition~\ref{prop:cluster-point}.
  \item[$(\ref{tc2})\Rightarrow(\ref{tc3}).$] Assume curve-compactness.  Let $q$ be a definable type on $Z$.  We claim that $q$ specializes to some point in $Z$.  Let $\{C_y : y \in Y\}$ be the interpretable family of basic closed sets in $q$.  By Remark~\ref{basic-closed}, it suffices to find a point in $\bigcap_{y \in Y} C_y$.

By the implication $(\ref{tr1})\Rightarrow(\ref{tr2})$ of Theorem~\ref{tri} (or Theorem~\ref{tri1} in the interpretable case), there is a $\Gamma$-family $\{X_\gamma: \gamma\in \Gamma\}$ that refines $\{C_y : y\in Y\}$. Applying definable Skolem functions to the family $\{X_{\val(x)} : x\in K\setminus\{0\}\}$, there is an interpretable curve $f: K\setminus \{0\} \rightarrow Z$ such that $f(x)\in X_{\val(x)}$ for every $x$. Applying the definition of curve-compactness, let $g$ be a restriction of $f$ that converges to some point $z\in Z$. Clearly $z\in C_y$ for every $y\in Y$.   
  \item[$(\ref{tc3})\Rightarrow(\ref{tc4}).$] A direct consequence of
    Corollary~\ref{cor:sim-star-fact} (or Lemma~\ref{ssf2} in the
    interpretable case).
  \item[$(\ref{tc4})\Rightarrow(\ref{tc1}).$] Assume that $(Z,\tau)$ satisfies $(\ref{tc4})$ and let $\{C_y : y\in Y\}$ be an interpretable downward directed family of non-empty closed subsets of $Z$. In particular $\{C_y : y\in Y\}$ has the finite intersection property, and so it has a finite transversal $\{z_1,\ldots, z_n\} \subset K$. In other words, for every $y\in Y$ there exists some $i\leq n$ such that $C_y \in \tp(z_i/K)$. It follows from Lemma~\ref{lem:down-dir} that there is some $i\leq n$ with $z_i\in \bigcap_{y\in Y} C_y$.  \qedhere
  \end{description}
\end{proof}

\begin{definition}\label{dfn:comp}
An interpretable topological space $(Z,\tau)$ is \emph{definably compact} if the equivalent conditions of Theorem~\ref{thm:compactness} hold.
\end{definition}

Note that in the proof of Theorem~\ref{thm:compactness} we only apply definable Skolem functions in order to prove the implication (\ref{tc2})$\Rightarrow$(\ref{tc3}). The equivalence between (\ref{tc1}), (\ref{tc1.5}), (\ref{tc3}) and (\ref{tc4}) can be derived from Proposition~\ref{refine-g}, distality, and Fact~\ref{fact:simon-star} without additional $\pCF$ machinery.

\begin{remark} \label{rem:pq2}
  Let $(Z,\tau)$ be a definable topological space in $K$, with $Z\subseteq K^n$. Then the five conditions in
  Theorem~\ref{thm:compactness} are equivalent to the following:
  \begin{enumerate}
    \setcounter{enumi}{5}
  \item \label{tc5} Any definable family of non-empty closed sets $\{C_y\}_{y \in Y}$ in
    $Z$ with the $(m,2n)$-property, for some $m \ge 2n$,
    has a finite transveral.
\end{enumerate}

To see this note that, by Remark~\ref{rem:pq}, Condition (\ref{tc3}) in Theorem~\ref{thm:compactness} implies (\ref{tc5}). Furthermore, Condition (\ref{tc5}) is stronger than Condition (\ref{tc4}) in Theorem~\ref{thm:compactness}, because the finite intersection property is stronger than the $(p,q)$-property for any $p \ge q$.
\end{remark}

\begin{remark}
Definable Skolem functions yield a definable analogue of the classical characterization of compactness in terms of nets (every net has a convergent subnet), where a definable net is understood to be an interpretable map from a definable directed set into an interpretable topological space.  That is, in a structure with definable Skolem functions (e.g.\@ a $p$-adically closed field) one may show, by direct adaptation of the proof of the classical equivalence, that an interpretable topological space is directed-compact if and only if every definable net in it has a convergent definable subnet. See~\cite[Corollary 44]{atw1} for a proof of this fact in the o-minimal group setting.
\end{remark}

\begin{theorem} \label{qp-case}
An interpretable topological space $(Z,\tau)$ in $\Qp$ is definably compact if and only if it is compact. 
\end{theorem}
\begin{proof}
  If $Z$ is compact, then it is clearly directed-compact,
  type-compact, and curve-compact.  Conversely, suppose $Z$ is
  type-compact.  Fix an interpretable basis for $\tau$.  Note that any intersection of closed subsets of $Z$ can be rewritten as an intersection of basic
  closed sets.  Any family of basic closed sets
  $\mathcal{C}$ with the finite intersection property extends to a
  type $q\in S_Z(\Qp)$.  The type $q$ is definable by Fact~\ref{delon}
  (or Lemma~\ref{delon-lift} in the interpretable case).  By type-compactness, $q$ has a specialization $z\in Z$, which will
  satisfy $z\in \bigcap \mathcal{C}$.
\end{proof}

\begin{theorem}\label{compactness-def}
  Let $\{(X_b,\tau_b) : b \in Y\}$ be an interpretable family of interpretable
  topological spaces.  Then $\{b \in Y : X_b \text{ is
    definably compact}\}$ is interpretable.
\end{theorem}
\begin{proof}
First suppose that $K$ is a monster model $\Mm$.
Fix a sort $D$ in $\Mm^\eq$ such that $X_b \subseteq D$ for each $b$.  Identify $D^\df$ with $S^\df_D(\Mm)$.
  For each $b \in Y$, let $\{C_{b,i}\}_{i \in
    I_b}$ be a basis of closed sets for $X_b$.  Because $\{X_b\}_{b
    \in Y}$ is an interpretable family of interpretable topologies, we may
  assume that $\{I_b\}_{b \in Y}$ and $\{C_{b,i}\}_{b \in Y, \,i \in
    I_b}$ are interpretable families.

  By Theorem~\ref{thm:compactness}, $X_b$ is definably compact if and only if
  $X_b$ is type-compact, meaning that every $p \in S^{\df}_{X_b}(\Mm)$
  has a specialization.  By Remark~\ref{basic-closed}, $p$ specializes
  to $x \in X_b$ if and only if $x$ is in every basic closed set
  contained in $p$, i.e.,
  \begin{equation*}
    \forall i \in I_b ~ (C_{b,i} \in p \rightarrow x \in C_{b,i}).
  \end{equation*}
  Thus $X_b$ is definably compact if and only if
  \begin{equation*}
    \forall p \in X_b^{\df} ~ \exists x \in X_b ~ \forall i \in I_b ~
    (C_{b,i} \in p \rightarrow x \in C_{b,i}),
  \end{equation*}
  or equivalently
  \begin{equation*}
    \forall p \in D^\df ~ \left(X_b \in p \rightarrow \exists
    x \in X_b ~ \forall i \in I_b ~ (C_{b,i} \in p \rightarrow x \in
    C_{b,i}) \right).
  \end{equation*}
  This condition is first-order: the expressions $X_b \in p$ and
  $C_{b,i} \in p$ are first-order by Remark~\ref{deffam}, and the
  quantificiation over $D^{\df}$ is harmless by Remark~\ref{quant},
  \emph{because $D^{\df}$ is strictly pro-interpretable} by Fact~\ref{strict} (or Theorem~\ref{strict2} in the interpretable case).
  
  This completes the case when $K$ is a monster model.  For the general case, take a monster model $\Mm \succeq K$.  Let $\{(X_b(\Mm),\tau_b(\Mm)) : b \in Y(\Mm)\}$ denote the interpretable family of interpretable topological spaces in $\Mm$ defined by the same formulas as the original family $\{(X_b,\tau_b) : b \in Y\}$.  Definable compactness is preserved in elementary extensions.  (This is straightforward to see using the definitions of directed-compactness or curve-compactness.)  Therefore
  \[ (X_b(\Mm),\tau_b(\Mm)) \text{ is definably compact} \iff (X_b,\tau_b) \text{ is definably compact}\]
  for $b \in Y(K)$.  Let $Q$ be the set of $b \in Y(\Mm)$ such that $(X_b(\Mm),\tau_b(\Mm))$ is definably compact.  By the highly saturated case considered above, $Q$ is an interpretable subset of $Y(\Mm)$.  It is also $\Aut(\Mm/K)$-invariant, and therefore $K$-interpretable.  Then the set
  \[ \{b \in Y(K) : (X_b,\tau_b) \text{ is definably compact}\} = Q(K)\]
  is interpretable in the structure $K$.
\end{proof}

Theorem~\ref{compactness-def} is a substantial strengthening of \cite[Theorem~6.6]{johnson-admissible}, which proved the same result for the special case of ``strongly admissible'' interpretable topologies.
  
\section{Open problems}

A natural question to ask is whether the results of this paper
generalize to P-minimal fields \cite{p-min}.  The results of this
paper depend heavily on Sections~\ref{tools}--\ref{analysis}, which
are very specific to $\pCF$ and don't generalize in an obvious way to
P-minimal expansions of $\pCF$.

In an orthogonal direction, it would be interesting to see whether
some of the results in this paper generalize to other dp-minimal
valued fields, such as the fields of Laurent functions $\Rr((t))$ and
$\Qp((t))$.  Some things fail to generalize to ACVF, as discussed in
Section~\ref{sec:acvf}.  Perhaps the natural dividing line is
distality.  For example, does Theorem~\ref{tp-code} hold in distal
dp-minimal valued fields?

Beyond the valued field setting, we ask whether Theorem~\ref{tri} generalizes in the following sense: in distal dp-minimal structures, any definable family of sets that extends to a definable type is refined by a definable downward directed family in the type. This was proved for o-minimal structures in~\cite[Theorem A]{types-transversals}.
The reverse question, namely whether any definable downward directed family of sets extends to a definable type, is also open and would follow for all dp-minimal structures from a positive answer to Simon's conjecture~\cite[Conjecture 5.2]{simon15} that Fact~\ref{fact:simon-star} holds in the general dp-minimal setting.

Another open problem, discussed above, is whether $M$-definable types
can be lifted along interpretable surjections (Question~\ref{open-q}).

\appendix

\section{New proofs of known results}
In this appendix, we give self-contained proofs of three facts about
$\pCF$ that were used earlier in the paper:
\begin{enumerate}
\item Fact~\ref{delon}: the definability of types over $\Qp$, due to
  Delon \cite{delon}.
\item Fact~\ref{uddt}: the uniform definability of definable types in
  $\pCF$, due to Cubides Kovacsics and Ye \cite{cubides-ye}.
\item Fact~\ref{strict}: the strict pro-definability of the space of
  definable types in $\pCF$, due to Cubides Kovacsics, Hils, and Ye \cite{cubides-ye-hils}.
\end{enumerate}
Our proof of Fact~\ref{strict} is probably new.  In contrast, our
proofs of Fact~\ref{delon} and \ref{uddt} might be equivalent to the
original proofs.  Nevertheless, since the three facts follow easily
from the machinery in this paper, it seemed worthwhile to include the
proofs.

\begin{lemma} \label{qf-an2}
  Let $K, \Mm, a, K[x]_{<d}, I_d,\tau_d$ be as in Lemma~\ref{qf-an}.  The following are equivalent:
  \begin{enumerate}
  \item \label{2qa1} $\tp(a/K)$ is definable.
  \item \label{2qa2} Each $\tau_d$ is split.
  \item \label{2qa3} Each $\tau_d$ is definable.
  \end{enumerate}
\end{lemma}
\begin{proof}
  (\ref{2qa1})$\implies$(\ref{2qa2}): Lemma~\ref{qf-an}(\ref{qan4}).

  (\ref{2qa2})$\implies$(\ref{2qa3}): split VVS structures are definable.

  (\ref{2qa3})$\implies$(\ref{2qa1}): Suppose the $\tau_d$ are definable.  Note that the
  sets $I_d$ are also definable, as they are $K$-linear subspaces of
  finite-dimensional $K$-vector spaces.  Let $q = \tp(a/K)$.
  \begin{claim}
    Suppose $K' \succeq K$, and suppose $q', q'' \in S_n(K')$ are
    heirs of $q$.  Then $q'$ and $q''$ have the same quantifier-free
    part.
  \end{claim}
  \begin{claimproof}
    Let $I'_d \subseteq K'[x]_{<d}$ and $\tau'_d$ on $K'[x]_{<d}/I'_d$
    be derived from $q'$ the same way that $I_d$ and $\tau_d$ are
    derived from $q$.  Similarly, let $I''_d$ and $\tau''_d$ be
    derived from $q''$.  Because $q'$ is an heir of $q$, $I'_d$ and
    $\tau'_d$ must be definable, defined by the same $\Ldiv(K)$-formulas that
    define $I_d$ and $\tau_d$.  The same holds for $I''_d$ and
    $\tau''_d$, so $I'_d = I''_d$ and $\tau'_d = \tau''_d$.  By
    Lemma~\ref{qf-an}(\ref{qan2}), $q'$ and $q''$ have the same quantifier-free
    part.
  \end{claimproof}
  By the claim and Proposition~\ref{qf-count}, $q$ has at most
  $2^{\aleph_0}$-many heirs over any elementary extension $K' \succeq
  K$.  Therefore $q$ is definable.
\end{proof}

\begin{remark} \label{tcom2}
  Suppose $M \succeq \Qp$ and $b \in M^1$.  If $\val(b) \ge 0$, then
  there is $a \in \Qp$ such that $b - a$ is $\Qp$-infinitesimal, in
  the sense that $\val(b - a) > \Gamma(\Qp) = \Zz$.  This holds
  because $\Zz_p$ is compact.
\end{remark}

\begin{lemma} \label{split2}
  Suppose $M \succeq \Qp$ and $a$ is a tuple in $M$.
  Let $V \subseteq \Qp(a)$ be a finite-dimensional $K$-linear
  subspace, with the induced VVS structure.  Then $V$ is split.
\end{lemma}
Lemma~\ref{split2} is like Lemma~\ref{split}, but instead of assuming
$\tp(a/K)$ is definable, we assume $K = \Qp$.
\begin{proof}
  Like Lemma~\ref{split}, using Remark~\ref{tcom2} instead of
  Fact~\ref{tcom}.
\end{proof}
Delon's theorem follows easily:
\begin{theorem}[{= Fact~\ref{delon}}] \label{delon2}
  Every type over $\Qp$ is definable.
\end{theorem}
\begin{proof}
  Take $a \in M \succeq \Qp$.  We claim $\tp(a/\Qp)$ is definable.
  Let $I_d$ and $\tau_d$ be as in Lemmas~\ref{qf-an} and \ref{qf-an2}.
  By Lemma~\ref{qf-an2}, we must show that $\tau_d$ is split.  But
  $\tau_d$ is the induced VVS structure on the image of the map
  \begin{align*}
    \Qp[x]_{<d} & \to \Qp(a) \\
    P(x) &\mapsto P(a).
  \end{align*}
  The image is a finite-dimensional subspace of $\Qp(a)$, so the VVS
  structure is split by Lemma~\ref{split2}.
\end{proof}

\begin{lemma} \label{kn}
  Let $\mathcal{K}_n$ be the class of structures $(M,K,a)$ where $M
  \models \pCF$, $K \preceq M$, $a \in M^n$, and $\tp(a/K)$ is
  definable.  Then $\mathcal{K}_n$ is elementary.
\end{lemma}
\begin{proof}
  The difficulty is expressing that $\tp(a/K)$ is definable.  By
  Lemma~\ref{qf-an2}, it suffices to express the condition ``$\tau_d$
  is a split VVS structure'' for each $d$.  This is straightforward,
  if tedious.
\end{proof}
The result of Cubides Kovacsics and Ye follows directly:
\begin{theorem}[{= Fact~\ref{uddt}}] \label{uddt2}
  The theory $\pCF$ has uniform definability of definable types (UDDT).
\end{theorem}
\begin{proof}
  Fix an $\Ldiv$-formula $\varphi(x,y)$ and let $n = |x|$.  We must find an $\Ldiv$-formula
  $\psi(y,z)$ such that for any model $M \models T$ and definable type
  $p \in S_n(M)$, the set $\{b \in M : p(x) \vdash \varphi(x,b)\}$ is
  a $\psi$-set.  Let $T_n$ axiomatize the class $\mathcal{K}_n$ of
  Lemma~\ref{kn} in the language $\mathcal{L}'$ of structures
  $(M,K,a)$.  For each $\Ldiv$-formula $\psi$, let $\alpha_\psi$
  be the $\mathcal{L}'$-sentence saying that $\tp^{\varphi}(a/K)$ is
  defined by a $\psi$-formula, in the sense that
  \begin{equation*}
    \{b \in K : M \models \varphi(a,b)\} = \psi(K,c) \text{ for some
    } c \in K.
  \end{equation*}
  Every model of $T_n$ satisfies $\alpha_\psi$ for some $\psi$. By
  compactness, there are finitely many $\Ldiv$-formulas $\psi_1,\ldots,\psi_m$
  such that $T_n \vdash \bigvee_{i = 1}^m \alpha_{\psi_i}$.  This
  means that if $K \models \pCF$ and $\tp(a/K)$ is definable, then
  $\tp^\varphi(a/K)$ is defined by a $\psi_i$-formula for some $i$.
  The usual coding tricks allow us to reduce to a single $\psi$,
  and then we have UDDT.
\end{proof}
Note that the proof of Theorem~\ref{uddt2} is really just a miniature version of the
argument in \cite{cubides-ye}.

Finally, we prove strict pro-interpretability of the space of
definable types, due to Cubides Kovacsics, Hils, and Ye \cite{cubides-ye-hils}.  Let $X$ be a definable
set and let $X^{\df}$ be the pro-interpretable set of definable types
as in Section~\ref{prodsec}.
\begin{theorem}[{= Fact~\ref{strict}}] \label{strict3}
  $X^{\df}$ is strictly pro-interpretable.
\end{theorem}
\begin{proof}
  We easily reduce to the case where $X$ is $\Mm^k$.  (If $X$ is a
  definable subset of $\Mm^k$, then $X^{\df}$ is relatively definable
  in $(\Mm^k)^{\df}$.)
  
  Let $\varphi_1,\ldots,\varphi_n$ be $\Ldiv$-formulas.  We must show that the
  image of the following map on $S^{\df}_k(\Mm)$ is definable (and not
  just type-definable) in $\Mm^\eq$:
  \begin{equation*}
    f(p) = (c_{p,\varphi_i} : 1 \le i \le n).
  \end{equation*}
  Take an $\Ldiv$-formula $\varphi$ such that (1) for $1 \le i \le n$, every
  $\varphi_i$-set is a $\varphi$-set, and (2) the negation of a
  $\varphi$-set is a $\varphi$-set.  Then the $\varphi$-definition of $p$
  determines the $\varphi_i$-definition of $p$, for $p \in
  S^{\df}_k(\Mm)$, and so the map $f$ factors through $p \mapsto
  c_{p,\varphi}$.  Replacing $\varphi_1,\ldots,\varphi_n$ with $\varphi$, we may assume $n = 1$ and we need to
  show definability of
  \begin{equation*}
    D := \{c_{p,\varphi} : p \in S^{\df}_k(\Mm)\}.
  \end{equation*}
  Let $\psi$ be the $\Ldiv$-formula uniformly defining $\varphi$, so that
  \[ \{b \in \Mm : \varphi(x,b) \in p(x)\} = \psi(\Mm,c_{p,\varphi}) \]
  \begin{claim}
    $c \in D$ if and only if the following two conditions hold.
    \begin{enumerate}
        \item For any
    $b, b' \in \Mm^y$ such that $\varphi(\Mm,b)$ and $\varphi(\Mm,b')$
    are complementary, exactly one of $b, b'$ is in $\psi(\Mm,c)$.
    \item The family of sets $\{\varphi(\Mm,b) : b \in
    \psi(\Mm,c)\}$ is refined by a definable downward directed family
    of non-empty sets.
    \end{enumerate}
  \end{claim}
  \begin{claimproof}
    Suppose $c \in D$, so $c = c_{p,\varphi}$ for some $p$.  If $\varphi(\Mm,b)$ is
    complementary to $\varphi(\Mm,b')$, then exactly one of
    $\varphi(x,b), \varphi(x,b')$ is in $p(x)$, and so exactly one of
    $b, b'$ is in $\psi(\Mm,c)$.  Moreover, the family of
    sets $\{\varphi(\Mm,b) : b \in \psi(\Mm,c)\}$ extends to the
    definable type $p$, and therefore is refined by a definable
    downward directed family of non-empty sets by the implication
    (\ref{tr1})$\implies$(\ref{tr3}) of Theorem~\ref{tri}.

    Conversely, suppose $c$ satisfies the two listed conditions.
    There is a definable type $p \in S^{\df}_k(\Mm)$ extending the
    family $\{\varphi(\Mm,b) : b \in \psi(\Mm,c)\}$ by the implication
    (\ref{tr3})$\implies$(\ref{tr1}) of Theorem~\ref{tri}.  Then
    \begin{equation*}
      \psi(\Mm,c) \subseteq \{b \in \Mm^y : \varphi(x,b) \in p(x)\}.
    \end{equation*}
    If equality holds, then $c = c_{\varphi,p} \in D$ as desired.
    Otherwise, take $b \notin \psi(\Mm,c)$ such that $\varphi(x,b) \in
    p(x)$.  Take $b'$ such that $\varphi(\Mm,b')$ is the complement of
    $\varphi(\Mm,b)$.  By assumption, exactly one of $b, b'$ is in
    $\psi(\Mm,c)$, and so $b' \in \psi(\Mm,c)$, implying that
    $\varphi(x,b') \in p(x)$.  Then $p(x)$ contains the contradictory
    formulas $\{\varphi(x,b),\varphi(x,b')\}$, a contradiction.
  \end{claimproof}
  The conditions in the claim are $\vee$-definable, so $D$ is
  $\vee$-definable in $\Mm^\eq$.  On the other hand, $D$ is the image
  of the pro-definable set $X^{\df}$ under a coordinate projection, so
  $D$ is type-definable in $\Mm^\eq$.  By compactness, $D$ is
  definable in $\Mm^\eq$.
\end{proof}

\section{An application of the $(p,q)$-theorem}
In this appendix we show how the Alon-Kleitman-Matousek $(p,q)$-theorem~\cite[Theorem 4]{matousek04} can be used to study the notion of dividing in the NIP setting, and in particular in $\pCF$, with applications to Remarks~\ref{rem:pq} and \ref{rem:pq2}. While the $(p,q)$-theorem has already seen strong applications in the NIP setting by Simon~\cite{cher_sim_15, simon15} and Chernikov~\cite{cher_sim_15}, they rely on a weaker form of the theorem (in terms of dual VC-dimension in place of VC-codensity) that does not yield our results.

Let $\mathcal{F}$ be a family of 
subsets of some set $U$.
For any finite $\mathcal{S} = \{X_1,\ldots,X_n\} \subseteq
\mathcal{F}$, let $\sim_{\mathcal{S}}$ be the equivalence relation on
$U$ defined by
\begin{equation*}
  x \sim_{\mathcal{S}} y \iff (\forall i) ~ (x \in X_i \iff y \in X_i).
\end{equation*}
We call each equivalence class of $\sim_{\mathcal{S}}$ a \emph{Boolean atom} of $\mathcal{S}$.
One defines
$\pi^*_{\mathcal{F}}(n)$ to be the maximum $k$ ($\leq 2^n$) such that there exist
$X_1,\ldots,X_n \in \mathcal{F}$ with $k$ Boolean atoms.  The
function $\pi^*_{\mathcal{F}}(-)$ is called the \emph{dual shatter
function} of $\mathcal{F}$.  The \emph{VC-codensity} of $\mathcal{F}$
is the infimum of positive real numbers $\alpha$ such that
$\pi^*_{\mathcal{F}}(n)$ is $O(n^\alpha)$, or $\infty$ if no such
$\alpha$ exists.  We write the VC-codensity of $\mathcal{F}$ as
$\vc^*(\mathcal{F})$.
\begin{definition}
  Let $p \ge q \ge 1$ be integers.  The family $\mathcal{F}$ has the
  \emph{$(p,q)$-property} if $\varnothing \notin \mathcal{F}$ and, for any $p$ distinct sets in $\mathcal{F}$, there
  are $q$ among them with non-empty intersection.
\end{definition}
\begin{fact}[{Alon-Kleitman-Matou\v{s}ek $(p,q)$-theorem~\cite[Theorem 4]{matousek04}}]
  \label{pq-fact}
  Fix $p \ge q$.  Suppose $\vc^*(\mathcal{F}) < q$.  Then there is an
  integer $N$ such that that, if $\mathcal{S}$ is a finite subfamily of
  $\mathcal{F}$ with the $(p,q)$-property, then there exists a set of size $N$ intersecting every set in $\mathcal{S}$.
\end{fact}

If $M$ is an $L$-structure and $\varphi(x,y)$ is an $L(M)$-formula, then
$\vc^*(\varphi)$ denotes $\vc^*(\mathcal{F})$ where $\mathcal{F} =
\{\varphi(M,b) : b \in M^y\}$. The formula $\varphi(x,y)$ is NIP iff
$\vc^*(\varphi) < \infty$ \cite[Section 6.1]{NIPguide}.  Note that if $\varphi(x,y)$ has no parameters then $\vc^*(\varphi)$ is determined by the theory of $M$.

Suppose $M$ sits inside a monster model $\Mm$ and $b \in \Mm^{y}$.  By
unwinding the definitions~\cite[Definition 2.8]{kap-cher12}, we see that $\varphi(x,b)$ $k$-divides over
$M$ if and only if,
for every $m \ge k$, the family $\{\varphi(\Mm,b') : b'
\equiv_M b\}$ does not have the $(m,k)$-property.  Thus, $\varphi(x,b)$
does \emph{not} $k$-divide over $M$ if and only if
there is $m \ge k$
such that $\{\varphi(\Mm,b') : b' \equiv_M b\}$ has the $(m,k)$-property.

The following lemma, which we prove using Fact~\ref{pq-fact}, is an improvement of~\cite[Lemma 2.4]{simon15}, which states similar results in terms of dual VC-dimension in place of VC-codensity. 

\begin{lemma} \label{2.4like}
  Let $M, \Mm, \varphi(x,y)$ be as above.  Fix an integer $k >
  \vc^*(\varphi)$.  Let $D \subseteq \Mm^y$ be type-definable over
  $M$. The following are equivalent:
  \begin{enumerate}
  \item \label{xyz1} For any $b \in D$, the formula $\varphi(x,b)$ does not $k$-divide
    over $M$.
  \item \label{xyz2} There is $m \ge k$ such that the family $\{\varphi(\Mm,b) : b \in D\}$
    has the $(m,k)$-property.
  \item \label{xyz3} For any $b \in D$, the formula $\varphi(x,b)$ does not divide
    over $M$.
  \item \label{xyz4} For any $k'$, there is $m' \ge k'$ such that $\{\varphi(\Mm,b) :
    b \in D\}$ has the $(m',k')$-property.
  \end{enumerate}
\end{lemma}
\begin{proof}
  The implication (\ref{xyz4})$\implies$(\ref{xyz3}) follows by unwinding the definitions as above and the implication (\ref{xyz3})$\implies$(\ref{xyz1}) is trivial.  The implication (\ref{xyz1})$\implies$(\ref{xyz2}) follows by a standard compactness and Ramsey argument as in the proof of \cite[Lemma~2.4, $(3)\Rightarrow(2)$]{simon15}.  In more detail, suppose (\ref{xyz2}) fails. If $\varphi(\Mm,b)=\varnothing$ for some $b\in D$ then $\varphi(x,b)$ $k$-divides.  Otherwise, for any $m \ge k$ let $b_1, \ldots, b_m \in D$ be such that $\{\varphi(x,b_i) : 1 \le i \le m\}$ is $k$-inconsistent.  By compactness and Ramsey (or really, compactness and the pigeonhole principle), there are $b_1, b_2, b_3, \ldots \in D$ such that $\{\varphi(x,b_i): 1 \le i < \omega\}$ is $k$-inconsistent and $b_i \equiv_M b_j$ for any $i,j$.  Then $\varphi(x,b_1)$ $k$-divides over $M$.  
  
  Finally, we show (\ref{xyz2})$\implies$(\ref{xyz4}).  Let $\mathcal{F}$ be the family
  $\{\varphi(\Mm,b) : b \in D\}$.  Then $\vc^*(\mathcal{F}) \le
  \vc^*(\varphi) < k$.  Let $m$ be as in (\ref{xyz2}).  Then $\mathcal{F}$ and
  all its finite subfamilies have the $(m,k)$ property.  Let $N$ be
  given by Fact~\ref{pq-fact}.  Then for any finite subfamily
  $\mathcal{S} \subseteq \mathcal{F}$, there is a set of size $N$
  intersecting every element of $\mathcal{S}$. Fix any $k'$ and let $m' = (k'-1)N+1$.  We claim that $\mathcal{F}$ has the
  $(m',k')$-property.  Let $\mathcal{S} \subseteq \mathcal{F}$ be a
  subfamily of size $m'$.  Then there is a set $X$ of size $N$
  intersecting every element of $\mathcal{S}$.  As $|\mathcal{S}| >
  (k'-1)N$, by the pigeonhole principle there is some point $p \in X$
  such that at least $k'$ elements of $\mathcal{S}$ contain $p$.  That
  is, there are distinct $Y_1,\ldots,Y_{k'} \in \mathcal{S}$ all
  containing $p$.  Then $\bigcap_{i = 1}^{k'} Y_i \supseteq \{p\} \ne
  \varnothing$, proving the $(m',k')$-property.
\end{proof}

\begin{corollary} \label{low}
  If $M \preceq \Mm$ and $b \in \Mm$ and $\vc^*(\varphi) < k$, then $\varphi(x,b)$ divides
  over $M$ iff $\varphi(x,b)$ $k$-divides over $M$.
\end{corollary}
\begin{proof}
  Apply Lemma~\ref{2.4like} with $D$ equal to the set of realizations
  of $\tp(b/M)$.
\end{proof}
Corollary~\ref{low} is related to ``lowness''; see
\cite[Proposition~5.50]{NIPguide}.

\begin{corollary}\label{cor:pq-divide}
  Let $\varphi(x,y), \psi(y)$ be $L(M)$-formulas.  Suppose that
  $\{\varphi(M,b) : b \in \psi(M)\}$ has the $(m,k)$-property for some
  $m \ge k > \vc^*(\varphi)$.  Then $\varphi(x,b)$ does not divide
  over $M$ for any $b \in \psi(\Mm)$.
\end{corollary}
\begin{proof}
  The $(m,k)$-property is expressed by a first-order formula, so
  $\{\varphi(\Mm,b) : b \in \psi(\Mm)\}$ has the $(m,k)$-property.
  Then Lemma~\ref{2.4like} shows that $\varphi(x,b)$ does not divide
  for any $b \in \psi(\Mm)$.
\end{proof}

In the specific case of a $p$-adically closed field, we know bounds on the VC-codensity of formulas in terms of the number of object variables. 

\begin{fact}[\cite{vc_density}, Theorem~1.2]\label{fct:vc}
Let $K$ be a model of $\pCF$ and $\varphi(x,y)$ be a $\Ldiv(K)$-formula. Then $\vc^*(\varphi) \le 2|x|-1$.
\end{fact}

Although Theorem~1.2 in~\cite{vc_density} is stated for $\Qp$, recall that the VC-codensity of a formula without paramters is determined by the underlying theory, and moreover note that, for any formula $\varphi(x,y)$, with $y=(y_1,y_2)$, and parameters $c\in K^{y_2}$, clearly $\vc^*(\varphi(x,y_1,c))\leq \vc^*(\varphi(x,y))$.

Fact~\ref{fct:vc} yields $p$-adic versions of the results in this appendix, by substituting $\vc^*(\varphi)$ with $2|x|-1$. In particular, the $p$-adic version of Corollary~\ref{cor:pq-divide} can be applied to the proof of Corollary~\ref{cor:sim-star-fact} to yield Remark~\ref{rem:pq}, which in turn is used to characterize $p$-adic definable compactness in Remark~\ref{rem:pq2}.

\bibliographystyle{alpha} \bibliography{references}{}

\end{document}